\documentclass{article}%
\usepackage{bbm}
\usepackage{epsfig}
\usepackage{graphics,graphicx,amssymb,amsmath,verbatim}
\usepackage{amssymb}
\usepackage{mathrsfs}
\usepackage{amsfonts}
\usepackage{graphicx}
\usepackage{amsmath}%
\providecommand{\U}[1]{\protect\rule{.1in}{.1in}}
\newtheorem{theorem}{Theorem}

\newtheorem{conjecture}{Conjecture}
\newtheorem{corollary}{Corollary}

\newtheorem{definition}{Definition}
\newtheorem{example}{Example}

\newtheorem{lemma}{Lemma}

\newtheorem{proposition}{Proposition}
\newtheorem{remark}{Remark}

\newenvironment{proof}[1][Proof]{\noindent\textbf{#1.} }{\ \rule{0.5em}{0.5em}}
\begin{document}

\title{Toward Solution of Matrix Equation $X=Af(X)B+C$}
\author{Bin Zhou\thanks{Bin Zhou and Guang-Ren Duan are with the Center for Control
Theory and Guidance Technology, Harbin Institute of Technology, Harbin,
150001, P. R. China.}\thinspace\ \thanks{Corresponding author. Email:
\texttt{binzhoulee@163.com; binzhou@hit.edu.cn.}}\thinspace\ \quad James
Lam\thanks{James Lam is with Department of Mechanical Engineering, University
of Hong Kong, Hong Kong (E-mail: \texttt{james.lam@hku.hk}).}\thinspace\ \quad
Guang-Ren Duan$^{\ast}$}
\date{}
\maketitle

\begin{abstract}
This paper studies the solvability, existence of unique solution, closed-form
solution and numerical solution of matrix equation $X=Af\left(  X\right)  B+C$
with $f\left(  X\right)  =X^{\mathrm{T}},$ $f\left(  X\right)  =\overline{X}$
and $f\left(  X\right)  =X^{\mathrm{H}},$ where $X$ is the unknown. It is
proven that the solvability of these equations is equivalent to the
solvability of some auxiliary standard Stein equations in the form of
$W=\mathcal{A}W\mathcal{B}+\mathcal{C}$ where the dimensions of the
coefficient matrices $\mathcal{A},\mathcal{B}$ and $\mathcal{C}$ are the same
as those of the original equation. Closed-form solutions of equation
$X=Af\left(  X\right)  B+C$ can then be obtained by utilizing standard results
on the standard Stein equation. On the other hand, some generalized Stein
iterations and accelerated Stein iterations are proposed to obtain numerical
solutions of equation equation $X=Af\left(  X\right)  B+C$. Necessary and
sufficient conditions are established to guarantee the convergence of the iterations.

\vspace{0.5cm}


\end{abstract}

\section{Introduction}

Matrix equations play a very important role in linear system theory. For
example, the continuous-time Lyapunov matrix equation $AX+XA^{\mathrm{T}}=-Q$
and the discrete-time Lyapunov matrix equation $AXA^{\mathrm{T}}-X=-Q$ (which
is also known as the Stein equation) have very important applications in the
stability analysis of continuous-time and discrete-time linear systems,
respectively (\cite{zdg96book}). A more general class of equations in the form
of $AX+XB=C,$ which is called the Sylvester matrix equation, can be used to
solve many control problems such as pole assignment, eigenstructure
assignment, robust pole assignment, and observer design (see, for example,
\cite{db81laa} and \cite{zd05laa}). Other matrix equations such as the coupled
Sylvester matrix equations and the Riccati equations have also found numerous
applications in control theory. For more related introduction, see
\cite{flw09laa}, \cite{lrt92laa}, \cite{lwzd10amc}, \cite{zd05laa} and the
references therein.

Because of the wide applications of linear equations in control theory, many
approaches have been established in the literature to provide both closed-form
and numerical solutions. Take the Sylvester matrix equation $AX+XB=C$ for
example. The controllability and observability matrices are used in
\cite{db81laa} to construct the closed-form solutions, and such a solution is
also obtained by the inversion of a square matrix in \cite{jameson68siam}.
Regarding the numerical methods for solving the Sylvester matrix equation, one
of the most efficient approaches is the Hessenberg-Schur form based algorithm
proposed in \cite{gnl79tac}. Besides, the iterative algorithm by using the
hierarchical identification principle proposed in \cite{dc05tac} and
\cite{dld08amc} is also shown to be effective in solving this class of
equations. A more general form of this class of equations is recently
considered in \cite{bourgeois11laa}. For more references on this topic, see
\cite{dh10laa}, \cite{dl08laa}, \cite{huang96laa}, \cite{jbilou10laa},
\cite{lr09laa}, \cite{tt10laa}, \cite{miller88}, \cite{wimmer89laa},
\cite{zld09aml} and the references given there.

As a generalization of the continuous-time Lyapunov matrix equation, the
following linear matrix equation%
\begin{equation}
AX+X^{\mathrm{T}}B=C, \label{teq70}%
\end{equation}
has received much attention in the literature over the past several decades.
For example, special cases of (\ref{teq70}) in the form of $AX\pm
X^{\mathrm{T}}A^{\mathrm{T}}=B$ were thoroughly studied in \cite{td11laa},
\cite{d07jcam} and \cite{hodges57ampa} where necessary and sufficient
conditions for the existence of solutions were obtained. Very recently, the
solvability and closed-form solutions of equation (\ref{teq70}) is studied in
\cite{pzw07jfi} by using the Moore-Penrose generalized inverse. As a more
general case, the linear equation
\begin{equation}
AX^{\mathrm{T}}B+EXF=-C, \label{teq75}%
\end{equation}
is investigated in \cite{lwzd10amc} and \cite{wang07AMC} where iterative
approaches are developed to obtain the numerical solutions.

In this paper, we study matrix equation%
\begin{equation}
X=Af\left(  X\right)  B+C, \label{teqG}%
\end{equation}
with $f\left(  X\right)  =X^{\mathrm{T}},f\left(  X\right)  =\overline{X}$ and
$f\left(  X\right)  =X^{\mathrm{H}},$ and $A,B,C$ being known matrices. If
$f\left(  X\right)  =X^{\mathrm{T}},$ namely, $X=AX^{\mathrm{T}}B+C,$ then
equation (\ref{teqG}) can be regarded as a special case of (\ref{teq75}). If
$f\left(  X\right)  =\overline{X}$, equation (\ref{teqG}) becomes
$X=A\overline{X}B+C$ which was first studied in \cite{jw03laa}. Several
problems will be studied for equation (\ref{teqG}). These problems include
solvability, existence of unique solution, closed-form solution and numerical
solution. The key idea is to transform this class of equations into the
standard Stein equation (see, for example, \cite{ag10laa})%
\begin{equation}
W=\mathcal{A}W\mathcal{B}+\mathcal{C} \label{teqGt}%
\end{equation}
where $\mathcal{A},\mathcal{B}$ and $\mathcal{C}$ are functions of $A,B,C,$
and $W$ is unknown$.$ It is shown that the solvability and the existence of
unique solution of the original equation (\ref{teqG}) is equivalent to the
solvability and the existence of unique solution of equation (\ref{teqGt}).
Moreover, we have proven that if the general solution (denoted by
$\boldsymbol{W}$) of equation (\ref{teqGt}) has been obtained, then the
general solution (denoted by $\boldsymbol{X}$) of the original equation
(\ref{teqG}) can be obtained via the formulation%
\[
\boldsymbol{X}=\frac{1}{2}\left(  \boldsymbol{W}+Af\left(  \boldsymbol{W}%
\right)  B+C\right)  .
\]
Regarding numerical solutions, we extend our early results on Smith iteration
for standard Stein equation given in \cite{zld09aml} to equation (\ref{teqG})
and necessary and sufficient conditions are obtained to guarantee the
convergence of the iterates. Different from the real representation based
approach for solving equation $X=A\overline{X}B+C$ developed in \cite{jw03laa}
where the matrix equation is transformed into a standard Stein equation whose
coefficient matrices have dimensions twice of the original coefficient
matrices, here the dimensions of the coefficient matrices of the auxiliary
matrix equation (\ref{teqGt}) is the same as the dimensions of coefficient
matrices of the original equation (\ref{teqG}).

The rest of this paper is organized as follows: Some preliminary results in
linear algebra that will be used in this paper is recalled in Section
\ref{tsec1}. A basic result regarding the solution of linear equation is
developed in Section \ref{tsec1a}, which plays an important role in the
development of this paper. Then, for $f\left(  X\right)  =X^{\mathrm{T}%
},f\left(  X\right)  =\overline{X}$ and $f\left(  X\right)  =X^{\mathrm{H}},$
the matrix equation in (\ref{teqG}) is, respectively, considered\ in Section
\ref{tsec2}, Section \ref{tsec3} and Section \ref{tsec4}. A more general case
is treated in Section \ref{tsec7}. The paper is finally concluded in Section
\ref{tsec5}.

\section{\label{tsec1}Notations and Preliminaries}

\subsection{Notations and Some Standard Results}

In this paper, for a matrix $A$, we use $A^{\mathrm{T}},A^{\mathrm{H}%
},\overline{A},\lambda\left\{  A\right\}  ,\rho\left(  A\right)
,\mathrm{rank}\left(  A\right)  $ and $\det\left(  A\right)  $ to denote
respectively the transpose, the conjugated transpose, the conjugate, the
spectrum, the spectral radius, the rank and the determinant of $A.$ The symbol
$\otimes$ denotes the Kronecker product. For an $m\times n$ matrix $R=\left[
r_{ij}\right]  $, the so-called column stretching function $\mathrm{vec}%
\left(  \cdot\right)  $ is defined as
\[
\mathrm{vec}\left(  R\right)  =\left[
\begin{array}
[c]{ccccccc}%
r_{11} & \cdots & r_{m1} & \cdots & r_{1n} & \cdots & r_{mn}%
\end{array}
\right]  ^{\mathrm{T}}.
\]
For three matrices $A,X$ and $B$ with appropriate dimensions, we have the
following well-known results related to the column stretching function:%
\begin{equation}
\mathrm{vec}\left(  AXB\right)  =\left(  B^{\mathrm{T}}\otimes A\right)
\mathrm{vec}\left(  X\right)  . \label{axb1}%
\end{equation}

We then can introduce the following standard result.

\begin{lemma}
\label{tlm0}(\cite{cc02book} and \cite{lwzd10amc}) Let $X\in\mathbf{C}%
^{m\times n}$ be any matrix. Then
\[
\mathrm{vec}\left(  X^{\mathrm{T}}\right)  =P_{\left(  m,n\right)
}\mathrm{vec}\left(  X\right)  ,
\]
where $P_{\left(  m,n\right)  }$ is uniquely determined by the integers $m$
and $n.$ Moreover, the matrix $P_{\left(  m,n\right)  }$ has the following properties.

\begin{enumerate}
\item For two arbitrary integers $m$ and $n$, $P_{\left(  m,n\right)  }$ has
the following explicit form%
\[
P_{\left(  m,n\right)  }=\left[
\begin{array}
[c]{cccc}%
E_{11}^{\mathrm{T}} & E_{12}^{\mathrm{T}} & \cdots & E_{1n}^{\mathrm{T}}\\
E_{21}^{\mathrm{T}} & E_{22}^{\mathrm{T}} & \cdots & E_{2n}^{\mathrm{T}}\\
\vdots & \vdots & \ddots & \vdots\\
E_{m1}^{\mathrm{T}} & E_{m2}^{\mathrm{T}} & \cdots & E_{mn}^{\mathrm{T}}%
\end{array}
\right]  \in\mathbf{R}^{mn\times mn},
\]
where $E_{ij},i=1,2,\ldots,m,j=1,2,\ldots,n$ is an $m\times n$ matrix with the
element at position $\left(  i,j\right)  $ being 1 and the others being 0.

\item For two arbitrary integers $m$ and $n$, $P\left(  m,n\right)  $ is a
unitary matrix, \textit{that is}
\[
P_{\left(  m,n\right)  }P_{\left(  m,n\right)  }^{\mathrm{T}}=P_{\left(
m,n\right)  }^{\mathrm{T}}P_{\left(  m,n\right)  }=I_{mn}.
\]

\item For two arbitrary integers $m$ and $n$, there holds $P_{\left(
m,n\right)  }=P_{\left(  n,m\right)  }^{\mathrm{T}}=P_{\left(  n,m\right)
}^{-1}.$

\item Let $m,n,p$ and $q$ be four integers and $A\in\mathbf{C}^{m\times
n},B\in\mathbf{C}^{p\times q}.$ Then%
\begin{equation}
P_{\left(  m,p\right)  }\left(  B\otimes A\right)  =\left(  A\otimes B\right)
P_{\left(  n,q\right)  }. \label{teq98}%
\end{equation}

\end{enumerate}
\end{lemma}

Let $A\in\mathbf{C}^{m\times n}$ be any given matrix, the generalized inverse
of $A$, denoted by $A^{-}$, is such that%
\[
AA^{-}A=A.
\]
Then we can recall the following standard result regarding solutions of linear equation.

\begin{lemma}
\label{tlm13}Let $A\in\mathbf{C}^{m\times n}$ and $b\in\mathbf{C}^{m}$ be
given. Then the linear equation%
\begin{equation}
Ax=b \label{teq0}%
\end{equation}
is solvable if and only if there exists a $A^{-}$ such that $AA^{-}b=b$.
Moreover, if it is solvable and denote $x^{\dag}$ as one of its solutions,
then the general solution of equation (\ref{teq0}) is given by%
\[
x=x^{\dag}+\left(  I_{n}-A^{-}A\right)  t,
\]
where $t\in\mathbf{C}^{m}$ is an arbitrary vector.
\end{lemma}

The following result is the well-known Schur complement.

\begin{lemma}
\label{tlm10}(\cite{cc02book}) Let $X_{1},X_{2},X_{12}$ and $X_{21}$ be some
matrices with appropriate dimensions. Assume that $X_{1}$ and $X_{2}$ are
nonsingular. Then,%
\begin{align*}
\det\left[
\begin{array}
[c]{cc}%
X_{1} & X_{12}\\
X_{21} & X_{2}%
\end{array}
\right]   &  =\det\left(  X_{1}\right)  \det\left(  X_{2}-X_{21}X_{1}%
^{-1}X_{12}\right) \\
&  =\det\left(  X_{2}\right)  \det\left(  X_{1}-X_{12}X_{2}^{-1}X_{21}\right)
.
\end{align*}

\end{lemma}

Finally, we recall some facts about linear iteration.

\begin{definition}
(\cite{zld09aml}) For a linear iteration%
\begin{equation}
x_{k+1}=Mx_{k}+g,\text{ }g\in\mathbf{R}^{n}, \label{teqit}%
\end{equation}
where $M$ and $g$ are respectively constant matrix and vector and $\rho\left(
M\right)  <1,$ the asymptotic exponential convergence rate is $-\ln\left(
\rho\left(  M\right)  \right)  $.
\end{definition}

Basically, the smaller the number $\rho\left(  M\right)  ,$ the faster
convergence of the iteration (\ref{teqit}). See \cite{zld09aml} for explanations.

\subsection{Smith Iteration and Closed-Form Solutions of Standard Stein
Equation}

Consider the following standard Stein matrix equation%
\begin{equation}
X=AXB+C, \label{tstein}%
\end{equation}
where $A\in\mathbf{C}^{m\times m},B\in\mathbf{C}^{n\times n}$ and
$C\in\mathbf{C}^{m\times n}$ are known and $X\in\mathbf{C}^{m\times n}$ is to
be determined. The associated Smith iteration is defined as%
\begin{equation}
X_{k+1}=AX_{k}B+C,\text{ \ }\forall X_{0}\in\mathbf{C}^{m\times n}.
\label{tsmith}%
\end{equation}
It is known that the above Smith iteration converges to the unique solution of
equation (\ref{tstein}) for arbitrarily initial condition if and only if
$\rho\left(  A\right)  \rho\left(  B\right)  <1.$ Moreover, the asymptotic
exponential convergence rate is $-\ln\left(  \rho\left(  A\right)  \rho\left(
B\right)  \right)  $. Therefore, if $\rho\left(  A\right)  \rho\left(
B\right)  $ is very close to 1, the convergence is very slow. To improve the
convergence rate, the so-called Smith $\left(  l\right)  $ iteration can be
applied. The idea of Smith $\left(  l\right)  $ iteration is to combine $l$
steps of Smith iterations into one step. Notice that%
\begin{align}
X_{k+l}  &  =AX_{k+l-1}B+C\nonumber\\
&  =A\left(  AX_{k+l-2}B+C\right)  B+C\nonumber\\
&  =A^{2}X_{k+l-2}B^{2}+ACB+C\nonumber\\
&  =\cdots\nonumber\\
&  =A^{l}X_{k}B^{l}+\sum\limits_{i=0}^{l-1}A^{i}CB^{i}. \label{teq50}%
\end{align}
By denoting $X_{kl}=Y_{k},k=0,1,\ldots,$ we can rewrite (\ref{teq50}) as%
\begin{equation}
Y_{k+1}=A^{l}Y_{k}B^{l}+\sum\limits_{i=0}^{l-1}A^{i}CB^{i},\text{ }\forall
Y_{0}=X_{0}\in\mathbf{C}^{m\times n}. \label{teq30}%
\end{equation}
It is shown in \cite{zld09aml} that the above iteration converges to the
unique solution of (\ref{tstein}) for arbitrarily initial condition if and
only if $\rho\left(  A\right)  \rho\left(  B\right)  <1.$ Moreover, the
asymptotic exponential convergence rate is $-l\ln\left(  \rho\left(  A\right)
\rho\left(  B\right)  \right)  $. Therefore, the convergence rate of
(\ref{teq30}) increases as $l$ increases.

Though Smith $\left(  l\right)  $ iteration (\ref{teq30}) can improve the
convergence rate of the standard Smith iteration (\ref{tsmith}), it still
converges exponentially. In order to obtain super-exponential convergence rate
and by noting that the closed-form solution of equation (\ref{tstein}) is%
\[
X^{\star}=\sum\limits_{i=0}^{\infty}A^{i}CB^{i},\text{ \ \ }\rho\left(
A\right)  \rho\left(  B\right)  <1,
\]
the so-called $r$-Smith iteration can be constructed.

\begin{lemma}
\label{tlm5}(\cite{zld09aml}) Let $r\geq2$ be a prescribed positive integer.
Assume that $\rho\left(  A\right)  \rho\left(  B\right)  <1.$ Construct the
following iteration:%
\begin{equation}
X_{k+1}=\sum\limits_{i=0}^{r-1}A_{k}^{i}X_{k}B_{k}^{i},\text{ }A_{k+1}%
=A_{k}^{r},\text{ }B_{k+1}=B_{k}^{r},\text{ }k\geq0, \label{trsmith}%
\end{equation}
with initial condition $X_{0}=C,$ $A_{0}=A,$ and\ $B_{0}=B.$ Denote
$E_{k}=X^{\star}-X_{k}$ where $X^{\star}$ is the unique solution of
(\ref{tstein})$.$ Then $\lim_{k\rightarrow\infty}X_{k}=X^{\star}$ and%
\begin{equation}
X_{k}=\sum\limits_{i=0}^{r^{k}-1}A^{i}CB^{i},\text{ }E_{k}=A^{r^{k}}X^{\star
}B^{r^{k}},\text{ }k\geq0. \label{xk}%
\end{equation}

\end{lemma}

If $r=2,$ the $r$-Smith iteration reduces to the well-known Smith accelerative
iteration (\cite{smith68}). It is clear that the $r$-Smith iteration
(\ref{trsmith}) has a super-exponential convergence rate and is very effective
in obtaining numerical solutions of equation (\ref{tstein}) (\cite{zld09aml}).

In the following, we consider closed-form solutions of the Stein equation
(\ref{tstein}). Let the characteristic polynomial of $A$ be%
\[
h_{A}\left(  s\right)  =\det\left(  sI_{m}-A\right)  =s^{m}+\sum
\limits_{i=0}^{m-1}\alpha_{i}s^{i}.
\]
Then it is easy to verify that
\[
h_{A}\left(  s\right)  \triangleq\det\left(  I_{m}-sA\right)  =\alpha_{m}%
+\sum\limits_{i=1}^{m}\alpha_{m-i}s^{i},\text{ }\alpha_{m}=1.
\]

We then can cite the following result regarding closed-form solutions of
equation (\ref{tstein}).

\begin{lemma}
\label{tlm3}(\cite{jw03laa}) Let $A\in\mathbf{C}^{m\times m},B\in
\mathbf{C}^{n\times n}$ and $C\in\mathbf{C}^{m\times n}$ be given$.$ Then

\begin{enumerate}
\item If equation (\ref{tstein}) has a solution $X,$ then%
\[
Xh_{A}\left(  B\right)  =\sum\limits_{k=1}^{m}\sum\limits_{s=1}^{k}\alpha
_{k}A^{k-s}CB^{m-s}.
\]

\item If equation (\ref{tstein}) has a unique solution $X^{\star},$ then%
\[
X^{\star}=\left(  \sum\limits_{k=1}^{m}\sum\limits_{s=1}^{k}\alpha_{k}%
A^{k-s}CB^{m-s}\right)  \left(  h_{A}\left(  B\right)  \right)  ^{-1}.
\]

\end{enumerate}
\end{lemma}

\section{\label{tsec1a}A Basic Result on Linear Equation}

Before giving detailed discussion to equation (\ref{teqG}), we present a basic
result on linear equation. This result is essential in the development of this paper.

\begin{proposition}
\label{tpp1}Let $A\in\mathbf{C}^{n\times n}$ and $c\in\mathbf{C}^{n}$ be
given. Then the linear equation
\begin{equation}
\left(  I_{n}-A^{2}\right)  \omega=\left(  I_{n}+A\right)  c, \label{teq90}%
\end{equation}
is solvable with unknown $\omega$ if and only if the linear equation%
\begin{equation}
\left(  I_{n}-A\right)  x=c, \label{teq92}%
\end{equation}
is solvable with unknown $x.$ More specifically,

\begin{enumerate}
\item If $x$ is a solution of equation (\ref{teq92}), then $\omega=x$ is a
solution of equation (\ref{teq90}).

\item If $\boldsymbol{\omega}$ is the general solution of (\ref{teq90}), then
the general solution of equation (\ref{teq92}) is%
\begin{equation}
\boldsymbol{x}=\frac{1}{2}\left(  I_{n}+A\right)  \boldsymbol{\omega}+\frac
{1}{2}c. \label{teq91}%
\end{equation}

\item If equation (\ref{teq90}) has a unique solution $\omega^{\star}$, then
equation (\ref{teq92}) also has a unique solution $x^{\star}=\omega^{\star}.$
\end{enumerate}
\end{proposition}

\begin{proof}
\textit{Proof of Item 1}. This is trivial since equation (\ref{teq90}) can be
written as $\left(  I_{n}+A\right)  \left(  \left(  I_{n}-A\right)
\omega-c\right)  =0$.

\textit{Proof of Item 2}. Let $P$ be some matrix such that%
\begin{equation}
A=P\left[
\begin{array}
[c]{ccc}%
R & 0 & 0\\
0 & J & 0\\
0 & 0 & K
\end{array}
\right]  P^{-1},\text{ }R\in\mathbf{C}^{n_{0}\times n_{0}},\text{ }%
J\in\mathbf{C}^{n_{1}\times n_{1}},\text{ }K\in\mathbf{C}^{n_{2}\times n_{2}},
\label{teq85}%
\end{equation}
where all the eigenvalues of $J$ are 1, all the eigenvalues of $K$ are $-1$
and $\left\{  -1,1\right\}  \cap\lambda\left\{  R\right\}  =\varnothing.$
Clearly, $n_{0}+n_{1}+n_{2}=n.$ Then we can compute%
\begin{align}
I_{n}\pm A  &  =P\left[
\begin{array}
[c]{ccc}%
R_{\pm} & 0 & 0\\
0 & J_{\pm} & 0\\
0 & 0 & K_{\pm}%
\end{array}
\right]  P^{-1},\text{ }\label{teq82}\\
I_{n}-A^{2}  &  =P\left[
\begin{array}
[c]{ccc}%
R_{+}R_{-} & 0 & 0\\
0 & J_{-}J_{+} & 0\\
0 & 0 & K_{+}K_{-}%
\end{array}
\right]  P^{-1}, \label{teq83}%
\end{align}
where $R_{\pm}=I_{n_{0}}\pm R,J_{\pm}=I_{n_{1}}\pm J$ and $K_{\pm}=I_{n_{2}%
}\pm K.$ According to Lemma \ref{tlm13} and (\ref{teq83}), the general
closed-form solution of (\ref{teq90}) can be written as%
\begin{equation}
\boldsymbol{\omega}=P\left[
\begin{array}
[c]{c}%
0\\
J_{+}^{-1}\left(  I_{n_{1}}-J_{-}^{-}J_{-}\right)  s\\
K_{-}^{-1}\left(  I_{n_{1}}-K_{+}^{-}K_{+}\right)  t
\end{array}
\right]  +\omega^{\dag},\text{ }\forall s\in\mathbf{R}^{n_{1}},\forall
t\in\mathbf{R}^{n_{2}}, \label{teq84}%
\end{equation}
where $J_{-}^{-}$ and $K_{+}^{-}$ are any fixed generalized inversion of
$J_{-}$ and $K_{+}$, respectively, and $\omega^{\dag}$ is a special solution
of equation (\ref{teq90}), namely,%
\begin{equation}
\left(  I_{n}-A^{2}\right)  \omega^{\dag}=\left(  I_{n}+A\right)  c.
\label{teq26}%
\end{equation}
Let
\[
x^{\dag}\triangleq\frac{1}{2}\left(  I_{n}+A\right)  \omega^{\dag}+\frac{1}%
{2}c.
\]
Then, it follows from (\ref{teq26}) that%
\begin{align}
\left(  I_{n}-A\right)  x^{\dag}  &  =\left(  I_{n}-A\right)  \left(  \frac
{1}{2}\left(  I_{n}+A\right)  \omega^{\dag}+\frac{1}{2}c\right) \nonumber\\
&  =\frac{1}{2}\left(  I_{n}-A^{2}\right)  \omega^{\dag}+\frac{1}{2}\left(
I_{n}-A\right)  c\nonumber\\
&  =\frac{1}{2}\left(  I_{n}+A\right)  c+\frac{1}{2}\left(  I_{n}-A\right)
c\nonumber\\
&  =c, \label{teq94}%
\end{align}
namely, $x^{\dag}$ is a particular solution of equation (\ref{teq92}). Hence,
by using (\ref{teq85})--(\ref{teq84}), we have%
\begin{align}
\frac{1}{2}\left(  I_{n}+A\right)  \boldsymbol{\omega}+\frac{1}{2}c  &
=\frac{1}{2}P\left[
\begin{array}
[c]{ccc}%
R_{+} & 0 & 0\\
0 & J_{+} & 0\\
0 & 0 & K_{+}%
\end{array}
\right]  P^{-1}P\left[
\begin{array}
[c]{c}%
0\\
J_{+}^{-1}\left(  I_{n_{1}}-J_{-}^{-}J_{-}\right)  s\\
K_{-}^{-1}\left(  I_{n_{2}}-K_{+}^{-}K_{+}\right)  t
\end{array}
\right] \nonumber\\
&  \text{ \ \ \ }+\frac{1}{2}\left(  I_{n}+A\right)  \omega^{\dag}+\frac{1}%
{2}c\nonumber\\
&  =\frac{1}{2}P\left[
\begin{array}
[c]{c}%
0\\
\left(  I_{n_{1}}-J_{-}^{-}J_{-}\right)  s\\
K_{+}K_{-}^{-1}\left(  I_{n_{2}}-K_{+}^{-}K_{+}\right)  t
\end{array}
\right]  +x^{\dag}\nonumber\\
&  =\frac{1}{2}P\left[
\begin{array}
[c]{c}%
0\\
\left(  I_{n_{1}}-J_{-}^{-}J_{-}\right)  s\\
K_{-}^{-1}K_{+}\left(  I_{n_{2}}-K_{+}^{-}K_{+}\right)  t
\end{array}
\right]  +x^{\dag}\nonumber\\
&  =\frac{1}{2}P\left[
\begin{array}
[c]{c}%
0\\
\left(  I_{n_{1}}-J_{-}^{-}J_{-}\right)  s\\
K_{-}^{-1}\left(  K_{+}-K_{+}K_{+}^{-}K_{+}\right)  t
\end{array}
\right]  +x^{\dag}\nonumber\\
&  =P\left[
\begin{array}
[c]{c}%
0\\
\left(  I_{n_{1}}-J_{-}^{-}J_{-}\right)  s^{\prime}\\
0
\end{array}
\right]  +x^{\dag}, \label{teq88}%
\end{align}
where we have used the fact that $K_{+}=K_{+}K_{+}^{-}K_{+}$ and denoted
$s^{\prime}=\frac{1}{2}s$. On the other hand, in view of (\ref{teq94}) and by
using Lemma \ref{tlm13} again, the general solution of equation (\ref{teq92})
can be expressed as%
\begin{equation}
\boldsymbol{x}=P\left[
\begin{array}
[c]{c}%
0\\
\left(  I_{n_{1}}-J_{-}^{-}J_{-}\right)  s^{\prime}\\
0
\end{array}
\right]  +x^{\dag},\text{ }\forall s^{\prime}\in\mathbf{C}^{n_{1}}.
\label{teq89}%
\end{equation}
Comparing (\ref{teq88}) with (\ref{teq89}), we clearly see that
$\boldsymbol{x}$ in (\ref{teq91}) is the general solution of equation
(\ref{teq92}).

\textit{Proof of Item 3}. Equation (\ref{teq90}) has a unique solution
\[
\omega^{\star}=\left(  I_{n}-A^{2}\right)  ^{-1}\left(  I_{n}+A\right)
c=\left(  I_{n}-A\right)  ^{-1}c,
\]
if and only if $1\notin\lambda\{A^{2}\}$ which implies that $1\notin
\lambda\{A\},$ namely, equation (\ref{teq92}) has a unique solution $x^{\star
}=\left(  I_{n}-A\right)  ^{-1}c$. The proof is finished.
\end{proof}

We should point out that there is an important difference between Item 2 and
Item 1 of Proposition \ref{tpp1}: even though $\boldsymbol{x}$ is the general
solution of equation (\ref{teq92}), $\boldsymbol{\omega}=\boldsymbol{x}$ is
not the general solution of equation (\ref{teq90}). This can be observed from
(\ref{teq84}) and (\ref{teq89}). For illustration, we give an example.

\begin{example}
Let $A$ and $c$ in equations (\ref{teq90}) and (\ref{teq92}) be chosen as%
\[
A=\left[
\begin{array}
[c]{cc}%
1 & 0\\
0 & -1
\end{array}
\right]  ,\text{ \ }c=\left[
\begin{array}
[c]{c}%
0\\
1
\end{array}
\right]  .
\]
Then it is easy to see that the general solutions of equations (\ref{teq92})
and (\ref{teq90}) are, respectively,%
\[
\boldsymbol{x}=\left[
\begin{array}
[c]{c}%
t\\
\frac{1}{2}%
\end{array}
\right]  ,\text{ }\boldsymbol{\omega}=\left[
\begin{array}
[c]{c}%
t\\
s
\end{array}
\right]  ,\text{ \ }\forall t,s\in\mathbf{C}.
\]
It follows that there are two free parameters in $\boldsymbol{\omega}$ while
there is only one free parameter in $\boldsymbol{x}$ and thus
$\boldsymbol{\omega}=\boldsymbol{x}$ is not the general solution of equation
(\ref{teq90}).
\end{example}

Also, if equation (\ref{teq92}) has a unique solution, we cannot in general
conclude that equation (\ref{teq90}) has a unique solution. The simple
counterexample can be constructed by letting $-1\in\lambda\{A\}.$

\section{\label{tsec2}Equation $X=AX^{\mathrm{T}}B+C$}

In this section, we consider equation (\ref{teqG}) with $f\left(  X\right)
=X^{\mathrm{T}},$ namely, the following matrix equation%
\begin{equation}
X=AX^{\mathrm{T}}B+C, \label{teq}%
\end{equation}
where $A\in\mathbf{C}^{m\times n},B\in\mathbf{C}^{m\times n}$ and
$C\in\mathbf{C}^{m\times n}$ are some given matrices and $X\in\mathbf{C}%
^{m\times n}$ is a matrix to be determined.

\subsection{Solvability and Closed-Form Solutions}

By taking $\mathrm{vec}\left(  \cdot\right)  $ on both sides of (\ref{teq})
and using Lemma \ref{tlm0}, we have%
\begin{align}
\mathrm{vec}\left(  C\right)   &  =\left(  I_{mn}-\left(  B^{\mathrm{T}%
}\otimes A\right)  P_{\left(  m,n\right)  }\right)  \mathrm{vec}\left(
X\right) \nonumber\\
&  =\left(  P_{\left(  n,m\right)  }-B^{\mathrm{T}}\otimes A\right)
P_{\left(  m,n\right)  }\mathrm{vec}\left(  X\right)  . \label{teq32}%
\end{align}
Therefore, the following simple result can be obtained.

\begin{lemma}
\label{tlm7}Equation (\ref{teq}) is solvable if and only if%
\begin{equation}
\mathrm{rank}\left[
\begin{array}
[c]{cc}%
P_{\left(  n,m\right)  }-B^{\mathrm{T}}\otimes A & \mathrm{vec}\left(
C\right)
\end{array}
\right]  =\mathrm{rank}\left[
\begin{array}
[c]{c}%
P_{\left(  n,m\right)  }-B^{\mathrm{T}}\otimes A
\end{array}
\right]  . \label{teq31}%
\end{equation}
Moreover, it has a unique solution for arbitrary $C$ if and only if%
\begin{equation}
1\notin\lambda\left\{  \left(  B^{\mathrm{T}}\otimes A\right)  P_{\left(
m,n\right)  }\right\}  . \label{tcond1}%
\end{equation}

\end{lemma}

Notice that for test condition (\ref{tcond1}), we need to compute the
eigenvalues of a matrix having dimension $mn,$ which is demanding in practice.
In the following, we will provide an alternative condition for the solvability
of this equation.

\begin{theorem}
\label{tth1}Matrix equation (\ref{teq}) is solvable if and only if%
\begin{equation}
W=AB^{\mathrm{T}}WA^{\mathrm{T}}B+AC^{\mathrm{T}}B+C, \label{teq1}%
\end{equation}
is solvable with unknown $W.$ More specifically,

\begin{enumerate}
\item If $X$ is a solution of equation (\ref{teq}), then $W=X$ is a solution
of equation (\ref{teq1}).

\item If $\boldsymbol{W}$ is the general solution of equation (\ref{teq1}),
then the general solution to equation (\ref{teq}) is%
\begin{equation}
\boldsymbol{X}=\frac{1}{2}\left(  \boldsymbol{W}+A\boldsymbol{W}^{\mathrm{T}%
}B\right)  +\frac{1}{2}C. \label{teq96}%
\end{equation}

\item Equation (\ref{teq}) has a unique solution $X^{\star}$ if equation
(\ref{teq1}) has a unique solution $W^{\star}$, namely,%
\begin{equation}
\eta\gamma\neq1,\text{ \ \ }\forall\eta,\gamma\in\lambda\left\{
AB^{\mathrm{T}}\right\}  . \label{tcond}%
\end{equation}
Moreover, $X^{\star}=W^{\star}.$
\end{enumerate}
\end{theorem}

\begin{proof}
\textit{Proofs of Item 1 and Item 2}. By using (\ref{teq98}), we know that%
\begin{equation}
P_{\left(  m,n\right)  }\left(  B^{\mathrm{T}}\otimes A\right)  P_{\left(
m,n\right)  }=A\otimes B^{\mathrm{T}}. \label{teq104}%
\end{equation}
With this, by taking $\mathrm{vec}\left(  \cdot\right)  $ on both sides of
(\ref{teq1}) and denoting $\omega=\mathrm{vec}\left(  W\right)  ,$ we get%
\begin{align*}
\omega &  =\left(  B^{\mathrm{T}}\otimes A\right)  \mathrm{vec}\left(
B^{\mathrm{T}}WA^{\mathrm{T}}\right)  +\mathrm{vec}\left(  AC^{\mathrm{T}%
}B\right)  +\mathrm{vec}\left(  C\right) \\
&  =\left(  B^{\mathrm{T}}\otimes A\right)  \left(  A\otimes B^{\mathrm{T}%
}\right)  \omega+\left(  I_{mn}+\left(  B^{\mathrm{T}}\otimes A\right)
P_{\left(  m,n\right)  }\right)  \mathrm{vec}\left(  C\right) \\
&  =\left(  B^{\mathrm{T}}\otimes A\right)  P_{\left(  m,n\right)  }\left(
B^{\mathrm{T}}\otimes A\right)  P_{\left(  m,n\right)  }\omega+\left(
I_{mn}+\left(  B^{\mathrm{T}}\otimes A\right)  P_{\left(  m,n\right)
}\right)  \mathrm{vec}\left(  C\right) \\
&  =\mathit{\Upsilon}^{2}\omega+\left(  I_{mn}+\mathit{\Upsilon}\right)
\mathrm{vec}\left(  C\right)  ,
\end{align*}
where $\mathit{\Upsilon}=\left(  B^{\mathrm{T}}\otimes A\right)  P_{\left(
m,n\right)  }.$ The above equation can also be written as%
\begin{equation}
\left(  I_{mn}-\mathit{\Upsilon}^{2}\right)  \omega=\left(  I_{mn}%
+\mathit{\Upsilon}\right)  \mathrm{vec}\left(  C\right)  . \label{teq101}%
\end{equation}
Similarly, taking $\mathrm{vec}\left(  \cdot\right)  $ on both sides of
(\ref{teq}) and denoting $x=\mathrm{vec}\left(  X\right)  $ produce%
\begin{equation}
\left(  I_{mn}-\mathit{\Upsilon}\right)  x=\mathrm{vec}\left(  C\right)  ,
\label{teq102}%
\end{equation}
and taking $\mathrm{vec}\left(  \cdot\right)  $ on both sides of expression
(\ref{teq96}) yields%
\begin{equation}
\mathrm{vec}\left(  \boldsymbol{X}\right)  =\frac{1}{2}\left(  I_{mn}%
+\mathit{\Upsilon}\right)  \mathrm{vec}\left(  \boldsymbol{W}\right)
+\frac{1}{2}\mathrm{vec}\left(  C\right)  . \label{teq103}%
\end{equation}
We notice that equation (\ref{teq101}) is in the form of (\ref{teq90}),
equation (\ref{teq102}) is in the form of (\ref{teq92}) and expression
(\ref{teq103}) is similar to (\ref{teq91}). Then Item 1 and Item 2 follow from
Proposition \ref{tpp1} directly.

\textit{Proof of Item 3}. From the above proof and Proposition \ref{tpp1}, we
conclude that equation (\ref{teq}) has a unique solution if equation
(\ref{teq1}) has a unique solution which is equivalent to%
\begin{equation}
1\notin\lambda\left\{  \left(  A^{\mathrm{T}}B\right)  ^{\mathrm{T}}\otimes
AB^{\mathrm{T}}\right\}  =\lambda\left\{  B^{\mathrm{T}}A\otimes
AB^{\mathrm{T}}\right\}  . \label{teq2}%
\end{equation}

The remaining is to show that (\ref{tcond}) and (\ref{teq2}) are equivalent.
Without loss of generality, we assume that $n\geq m.$ Let $\lambda\left\{
AB^{\mathrm{T}}\right\}  =\{\eta_{1},\eta_{2},\ldots,\eta_{m}\}.$ Then%
\[
\lambda\left\{  B^{\mathrm{T}}A\right\}  =\{\eta_{1},\eta_{2},\ldots,\eta
_{m},0,0,\ldots,0\}.
\]
As a result, we have%
\[
\lambda\left\{  B^{\mathrm{T}}A\otimes AB^{\mathrm{T}}\right\}  =\left\{
\eta_{i}\eta_{j},0,0,\ldots,0\right\}  ,i,j\in\{1,2,\ldots,m\}.
\]
Therefore, (\ref{teq2}) is true if and only if (\ref{tcond}) is satisfied. The
proof is completed.
\end{proof}

One may ask whether the condition in (\ref{tcond}) is necessary for the
uniqueness of solution of equation (\ref{teq}). The answer is negative, as the
following example indicates.

\begin{example}
\label{texamp1}Consider a linear equation in the form of (\ref{teq}) with%
\begin{equation}
A=\left[
\begin{array}
[c]{cc}%
2 & 0\\
1 & \alpha
\end{array}
\right]  ,\text{ }B=\left[
\begin{array}
[c]{cc}%
1 & 0\\
0 & 1
\end{array}
\right]  , \label{teqab}%
\end{equation}
where $\alpha$ is to be determined. By direct computation, the characteristic
polynomial of $\left(  B^{\mathrm{T}}\otimes A\right)  P_{\left(  m,n\right)
}$ is $(s-2)(s-a)(s^{2}-2\alpha)=0.$ Therefore, according to (\ref{tcond1}),
the equation in (\ref{teq}) has a unique solution for arbitrary $C$ if and
only if $\alpha\neq1$ and $\alpha\neq\frac{1}{2}.$ On the other hand, as
$\lambda\left\{  B^{\mathrm{T}}A\right\}  =\{2,\alpha\},$ the condition in
(\ref{tcond}) is equivalent to $\alpha\neq\pm1$ and $\alpha\neq\frac{1}{2}.$
Hence, if $\alpha=-1,$ condition (\ref{tcond}) is not satisfied while the
equation (\ref{teq}) has a unique solution. This clearly implies that
(\ref{tcond}) is only a sufficient but not a necessary condition for the
existence of unique solution of equation (\ref{teq}) for arbitrary $C$.
\end{example}

Though Theorem \ref{tth1} only provides sufficient condition, the advantage is
that the condition in (\ref{tcond}) is easier to test than (\ref{tcond1})
because only the eigenvalues of a matrix of dimension $m$ is required.

\begin{remark}
If $n\leq m,$ we need only to compute the eigenvalues of a matrix of dimension
$n$ because condition (\ref{tcond}) can be replaced by $\eta\gamma
\neq1,\ \forall\eta,\gamma\in\lambda\left\{  BA^{\mathrm{T}}\right\}  .$
\end{remark}

By combining Lemma \ref{tlm3} and Theorem \ref{tth1}, we have the following
result regarding the closed-form solutions of equation (\ref{teq}).

\begin{theorem}
\label{tth6}Let $A\in\mathbf{C}^{m\times n},B\in\mathbf{C}^{m\times n}%
,C\in\mathbf{C}^{m\times n}$ and%
\[
h_{AB^{\mathrm{T}}}\left(  s\right)  =\det\left(  I_{m}-sAB^{\mathrm{T}%
}\right)  =\alpha_{m}+\sum\limits_{i=1}^{m}\alpha_{m-i}s^{i},\text{ }%
\alpha_{m}=1.
\]

\begin{enumerate}
\item If equation (\ref{teq}) has a solution $X,$ then%
\[
Xh_{AB^{\mathrm{T}}}\left(  A^{\mathrm{T}}B\right)  =\sum\limits_{k=1}^{m}%
\sum\limits_{s=1}^{k}\alpha_{k}\left(  AB^{\mathrm{T}}\right)  ^{k-s}\left(
AC^{\mathrm{T}}B+C\right)  \left(  A^{\mathrm{T}}B\right)  ^{m-s}.
\]

\item If (\ref{tcond}) is satisfied, then equation (\ref{teq}) has a unique
solution $X^{\star}\ $given by%
\[
X^{\star}=\left(  \sum\limits_{k=1}^{m}\sum\limits_{s=1}^{k}\alpha_{k}\left(
AB^{\mathrm{T}}\right)  ^{k-s}\left(  AC^{\mathrm{T}}B+C\right)  \left(
A^{\mathrm{T}}B\right)  ^{m-s}\right)  \left(  h_{AB^{\mathrm{T}}}\left(
A^{\mathrm{T}}B\right)  \right)  ^{-1}.
\]

\end{enumerate}
\end{theorem}

\subsection{Smith Iteration for Equation $X=AX^{\mathrm{T}}B+C$}

By recognizing the Smith iteration (\ref{tsmith}) for the standard Stein
equation (\ref{tstein}), we can construct the following Smith iteration for
matrix equation (\ref{teq}):%
\begin{equation}
X_{k+1}=AX_{k}^{\mathrm{T}}B+C,\text{ \ \ }\forall X_{0}\in\mathbf{C}^{m\times
n}. \label{teq21}%
\end{equation}
In the following, we will study the convergence of this iteration. First, we
introduce a lemma.

\begin{lemma}
\label{tlm1}Let $A,B\in\mathbf{C}^{m\times n}$ be given and $\mathit{\Upsilon
}=\left(  B^{\mathrm{T}}\otimes A\right)  P_{\left(  m,n\right)  }$. Then%
\[
\rho\left(  \mathit{\Upsilon}\right)  =\rho\left(  B^{\mathrm{T}}A\right)  .
\]

\end{lemma}

\begin{proof}
By definition of spectral radius and using (\ref{teq104}), we have%
\begin{align*}
\rho^{2}\left(  \mathit{\Upsilon}\right)   &  =\rho\left(  \mathit{\Upsilon
}^{2}\right) \\
&  =\rho\left(  \left(  \left(  B^{\mathrm{T}}\otimes A\right)  P_{\left(
m,n\right)  }\right)  \left(  \left(  B^{\mathrm{T}}\otimes A\right)
P_{\left(  m,n\right)  }\right)  \right) \\
&  =\rho\left(  \left(  B^{\mathrm{T}}\otimes A\right)  \left(  P_{\left(
m,n\right)  }\left(  B^{\mathrm{T}}\otimes A\right)  P_{\left(  m,n\right)
}\right)  \right) \\
&  =\rho\left(  \left(  B^{\mathrm{T}}\otimes A\right)  \left(  A\otimes
B^{\mathrm{T}}\right)  \right) \\
&  =\rho\left(  B^{\mathrm{T}}A\otimes AB^{\mathrm{T}}\right) \\
&  =\rho\left(  B^{\mathrm{T}}A\right)  \rho\left(  AB^{\mathrm{T}}\right) \\
&  =\rho^{2}\left(  B^{\mathrm{T}}A\right)  .
\end{align*}
The proof is completed.
\end{proof}

We then can prove the following result.

\begin{theorem}
\label{tth5}Assume that equation (\ref{teq}) has a unique solution $X^{\star}%
$. Then the Smith iteration (\ref{teq21}) converges to $X^{\star}$ for
arbitrary initial condition $X_{0}$ if and only if%
\begin{equation}
\rho\left(  B^{\mathrm{T}}A\right)  <1. \label{teqr}%
\end{equation}
Moreover, the asymptotic exponential convergence rate is $-\ln\left(
\rho\left(  B^{\mathrm{T}}A\right)  \right)  .$
\end{theorem}

\begin{proof}
Taking \textrm{vec}$\left(  \cdot\right)  $ on both sides of (\ref{teq21})
gives%
\[
\mathrm{vec}\left(  X_{k+1}\right)  =\mathit{\Upsilon}\mathrm{vec}\left(
X_{k}\right)  +\mathrm{vec}\left(  C\right)  ,
\]
which converges to a finite vector $\mathrm{vec}\left(  X_{\infty}\right)  $
independent of the initial condition if and only if $\rho\left(
\mathit{\Upsilon}\right)  <1.$ This is equivalent to (\ref{teqr}) in view of
Lemma \ref{tlm1}. Moreover, we have%
\[
X_{\infty}=AX_{\infty}^{\mathrm{T}}B+C.
\]
The proof is completed by noting the fact that equation (\ref{teq}) has only a
unique solution under the condition (\ref{teqr}).
\end{proof}

Notice that equation (\ref{teq1}) is a standard Stein equation whose
associated Smith iteration is%
\begin{equation}
X_{k+1}=AB^{\mathrm{T}}X_{k}A^{\mathrm{T}}B+AC^{\mathrm{T}}B+C,\text{
\ \ }\forall X_{0}\in\mathbf{C}^{m\times n}. \label{teq22}%
\end{equation}
It is trivial to show that this iteration converges to a constant matrix
independent of the initial condition $X_{0}$ if and only if (\ref{teqr}) is
satisfied. As (\ref{teqr}) implies (\ref{tcond}), we conclude from Theorem
\ref{tth1} that the iteration (\ref{teq22}) also converges to the unique
solution $X^{\star}$ to equation (\ref{teq}).

\begin{remark}
\label{trm1}In view of (\ref{teq50}), it is easy to see that iteration
(\ref{teq22}) can be regarded as Smith $\left(  2\right)  $ iteration
associated with the Smith iteration (\ref{teq21}). Therefore, by using a
similar technique as that used in (\ref{teq50}), for arbitrary integer $l>1,$
we can construct Smith $\left(  l\right)  $ iteration for equation
(\ref{teq}). The details are omitted for brevity.
\end{remark}

\begin{remark}
\label{trm2}Under condition (\ref{teqr}), it follows from Theorem \ref{tth1}
that the unique solution of equation (\ref{teq}) is%
\[
X^{\star}=\sum\limits_{i=0}^{\infty}\left(  AB^{\mathrm{T}}\right)
^{i}\left(  AC^{\mathrm{T}}B+C\right)  \left(  A^{\mathrm{T}}B\right)  ^{i}.
\]
Therefore, in view of Lemma \ref{tlm5}, for any integer $r\geq1$, we can
construct the following $r$-Smith iteration%
\[
X_{k+1}=\sum\limits_{i=0}^{r-1}\mathcal{A}_{k}^{i}X_{k}\mathcal{B}_{k}%
^{i},\text{ }\mathcal{A}_{k+1}=\mathcal{A}_{k}^{r},\text{ }\mathcal{B}%
_{k+1}=\mathcal{B}_{k}^{r},\text{ }k\geq0,
\]
with $X_{0}=AC^{\mathrm{T}}B+C,$ $\mathcal{A}_{0}=AB^{\mathrm{T}}\ $and
$\mathcal{B}_{0}=A^{\mathrm{T}}B.$ Consequently, we have $\lim_{k\rightarrow
\infty}X_{k}=X^{\star}.$ Actually, we have%
\[
X_{k}=\sum\limits_{i=0}^{r^{k}-1}\left(  AB^{\mathrm{T}}\right)  ^{i}\left(
AC^{\mathrm{T}}B+C\right)  \left(  A^{\mathrm{T}}B\right)  ^{i}.
\]

\end{remark}

\section{\label{tsec3}Revisit of Equation $X=A\overline{X}B+C$}

To study matrix equation $X=AX^{\mathrm{H}}B+C,$ we should first revisit the
following matrix equation
\begin{equation}
X=A\overline{X}B+C, \label{teq7}%
\end{equation}
where $A\in\mathbf{C}^{m\times m},B\in\mathbf{C}^{n\times n}$ and
$C\in\mathbf{C}^{m\times n}$ are given, and $X\in\mathbf{C}^{m\times n}$ is to
be determined. This equation was first studied in \cite{jw03laa}. To introduce
the results obtained in that paper, we need some preliminaries. For a complex
matrix $A=A_{1}+\mathrm{i}A_{2}$ where $\mathrm{i}=\sqrt{-1},$ the associated
real representation is%
\begin{equation}
A_{\sigma}=\left[
\begin{array}
[c]{cc}%
A_{1} & A_{2}\\
A_{2} & -A_{1}%
\end{array}
\right]  . \label{teqra}%
\end{equation}
The real representation has many good properties, one of which is the following.

\begin{lemma}
\label{tlm9}(Proposition 3.2 in \cite{jw03laa}) Let $A\in\mathbf{C}^{m\times
m}.$ If $\alpha\in\lambda\left\{  A_{\sigma}\right\}  $, then $\left\{
\pm\alpha,\pm\overline{\alpha}\right\}  \subset\lambda\left\{  A_{\sigma
}\right\}  .$
\end{lemma}

Based on the real representation, the following elegant result was established
in \cite{jw03laa} regarding solutions of equation (\ref{teq7}).

\begin{proposition}
\label{tlm4}(\cite{jw03laa}) The equation (\ref{teq7}) is solvable if and only
if the following equation%
\begin{equation}
Y=A_{\sigma}YB_{\sigma}+C_{\sigma}, \label{teq14}%
\end{equation}
is solvable with unknown $Y$. Moreover, let $Y$ be any real\footnote{Though it
is not mentioned in \cite{jw03laa} that $Y$ is real solution, it indeed should
be so.} solution of (\ref{teq14}), then the solution of equation (\ref{teq7})
is
\[
X=\frac{1}{4}\left[
\begin{array}
[c]{cc}%
I_{m} & \mathrm{i}I_{m}%
\end{array}
\right]  \left(  Y+Q_{m}YQ_{n}\right)  \left[
\begin{array}
[c]{c}%
I_{n}\\
\mathrm{i}I_{n}%
\end{array}
\right]  ,
\]
where%
\begin{equation}
Q_{s}=\left[
\begin{array}
[c]{cc}%
0 & I_{s}\\
-I_{s} & 0
\end{array}
\right]  . \label{tqs}%
\end{equation}

\end{proposition}

To use this result, we need to solve linear equation (\ref{teq14}) whose
dimension is twice of the original matrix equation (\ref{teq7}). In this
section, by using the idea used in Section \ref{tsec2}, we will give an
alternative criterion for the solvability and computation methods for equation
(\ref{teq7}).

\subsection{Some Useful Results}

In this subsection, we give some useful results that will be used later.

\begin{lemma}
\label{tlm2}Let $A=A_{1}+\mathrm{i}A_{2}\in\mathbf{C}^{m\times n},$ where
$A_{1},A_{2}\in\mathbf{R}^{m\times n},$ be any given complex matrices. Then%
\begin{equation}
A_{\phi}\triangleq\left[
\begin{array}
[c]{cc}%
A_{1} & -A_{2}\\
A_{2} & A_{1}%
\end{array}
\right]  =Z_{m}\left[
\begin{array}
[c]{cc}%
\overline{A} & 0\\
0 & A
\end{array}
\right]  Z_{n}^{\mathrm{H}}, \label{teq65}%
\end{equation}
where $Z_{n}$ is a unitary matrix defined as%
\[
Z_{n}=\frac{\sqrt{2}}{2}\left[
\begin{array}
[c]{cc}%
\mathrm{i}I_{n} & I_{n}\\
I_{n} & \mathrm{i}I_{n}%
\end{array}
\right]  .
\]
Moreover, for two arbitrary matrices $U\in\mathbf{C}^{m\times n}$ and
$V\in\mathbf{C}^{n\times q},$ there holds%
\begin{equation}
U_{\sigma}V_{\sigma}=\left(  U\overline{V}\right)  _{\phi}. \label{teq60}%
\end{equation}

\end{lemma}

\begin{proof}
Equation (\ref{teq65}) can be shown by direct manipulation. We next prove
(\ref{teq60}). By letting $U=U_{1}+\mathrm{i}U_{2}$ and $V=V_{1}%
+\mathrm{i}V_{2},$ where $U_{i}$ and $V_{i}$ are real matrices, we have%
\begin{align*}
U_{\sigma}V_{\sigma}  &  =\left[
\begin{array}
[c]{cc}%
U_{1} & U_{2}\\
U_{2} & -U_{1}%
\end{array}
\right]  \left[
\begin{array}
[c]{cc}%
V_{1} & V_{2}\\
V_{2} & -V_{1}%
\end{array}
\right] \\
&  =\left[
\begin{array}
[c]{cc}%
U_{1}V_{1}+U_{2}V_{2} & U_{1}V_{2}-U_{2}V_{1}\\
U_{2}V_{1}-U_{1}V_{2} & U_{1}V_{1}+U_{2}V_{2}%
\end{array}
\right] \\
&  =\left(  U_{1}V_{1}+U_{2}V_{2}+\mathrm{i}\left(  U_{2}V_{1}-U_{1}%
V_{2}\right)  \right)  _{\phi}\\
&  =\left(  \left(  U_{1}+\mathrm{i}U_{2}\right)  \left(  V_{1}-\mathrm{i}%
V_{2}\right)  \right)  _{\phi}\\
&  =\left(  U\overline{V}\right)  _{\phi},
\end{align*}
which is (\ref{teq60}). The proof is finished.
\end{proof}

Let $A=A_{1}+\mathrm{i}A_{2}\in\mathbf{C}^{m\times n}$ be any given matrix,
where $A_{1}$ and $A_{2}$ are real matrices. Then the linear mapping
$\varphi\left(  A\right)  :\mathbf{C}^{m\times n}\rightarrow\mathbf{R}^{2mn}$
is defined as%
\[
\varphi\left(  A\right)  =\left[
\begin{array}
[c]{c}%
\mathrm{vec}\left(  A_{1}\right) \\
\mathrm{vec}\left(  A_{2}\right)
\end{array}
\right]  .
\]
We notice that mapping $\varphi\left(  \cdot\right)  $ is bijective and denote%
\[
\varphi^{-1}\left(  c\right)  =\mathrm{vec}^{-1}\left(  c_{1}\right)
+\mathrm{vec}^{-1}\left(  c_{2}\right)  \mathrm{i},\text{ }c=\left[
\begin{array}
[c]{c}%
c_{1}\\
c_{2}%
\end{array}
\right]  \in\mathbf{R}^{2n\times1},\text{ }c_{1}\in\mathbf{R}^{n\times1}.
\]
Here we have assumed that $\mathrm{vec}^{-1}\left(  \cdot\right)  $ is defined
associated with $\mathrm{vec}\left(  \cdot\right)  $ in an obvious way. The
following lemma gives further properties of $\varphi\left(  A\right)  .$

\begin{lemma}
\label{tlm12}Let $A\in\mathbf{C}^{p\times m}$ and $B\in\mathbf{C}^{n\times q}$
be given. Then for any $X\in\mathbf{C}^{m\times n},$ there hold%
\begin{align}
\varphi\left(  A\overline{X}B\right)   &  =\left(  B^{\mathrm{T}}\otimes
A\right)  _{\sigma}\varphi\left(  X\right)  ,\label{teq77}\\
\varphi\left(  AXB\right)   &  =\left(  B^{\mathrm{T}}\otimes A\right)
_{\phi}\varphi\left(  X\right)  . \label{teq78}%
\end{align}

\end{lemma}

\begin{proof}
Let $A=A_{1}+\mathrm{i}A_{2},B=B_{1}+\mathrm{i}B_{2},X=X_{1}+\mathrm{i}X_{2}$
where $A_{i},B_{i}$ and $X_{i},i=1,2$ are real matrices. Then $Y=A\overline
{X}B$ where $Y=Y_{1}+\mathrm{i}Y_{2}\in\mathbf{C}^{m\times n}$ with $Y_{1}$
and $Y_{2}$ being real matrices, is equivalent to
\[
\left\{
\begin{array}
[c]{l}%
Y_{1}=\left(  A_{1}X_{1}+A_{2}X_{2}\right)  B_{1}-\left(  A_{2}X_{1}%
-A_{1}X_{2}\right)  B_{2},\\
Y_{2}=\left(  A_{2}X_{1}-A_{1}X_{2}\right)  B_{1}+\left(  A_{1}X_{1}%
+A_{2}X_{2}\right)  B_{2}.
\end{array}
\right.
\]
Taking \textrm{vec}$\left(  \cdot\right)  $ on both sides of the above two
equations, we get $\varphi\left(  Y\right)  =T\left(  A,B\right)
\varphi\left(  X\right)  ,$ where
\[
T\left(  A,B\right)  =\left[
\begin{array}
[c]{cc}%
B_{1}^{\mathrm{T}}\otimes A_{1}-B_{2}^{\mathrm{T}}\otimes A_{2} &
B_{1}^{\mathrm{T}}\otimes A_{2}+B_{2}^{\mathrm{T}}\otimes A_{1}\\
B_{1}^{\mathrm{T}}\otimes A_{2}+B_{2}^{\mathrm{T}}\otimes A_{1} &
B_{2}^{\mathrm{T}}\otimes A_{2}-B_{1}^{\mathrm{T}}\otimes A_{1}%
\end{array}
\right]  .
\]
In view of the real representation in (\ref{teqra}), we can see that%
\begin{align}
T\left(  A,B\right)   &  =\left(  \left(  B_{1}^{\mathrm{T}}\otimes
A_{1}-B_{2}^{\mathrm{T}}\otimes A_{2}\right)  +\mathrm{i}\left(
B_{1}^{\mathrm{T}}\otimes A_{2}+B_{2}^{\mathrm{T}}\otimes A_{1}\right)
\right)  _{\sigma}\nonumber\\
&  =\left(  \left(  B_{1}^{\mathrm{T}}+\mathrm{i}B_{2}^{\mathrm{T}}\right)
\otimes\left(  A_{1}+\mathrm{i}A_{2}\right)  \right)  _{\sigma}\nonumber\\
&  =\left(  B^{\mathrm{T}}\otimes A\right)  _{\sigma}, \label{teq45}%
\end{align}
which is (\ref{teq77}).

Next we show (\ref{teq78}). By using (\ref{teq60}), (\ref{teq77}) and by
definition of $\varphi\left(  \cdot\right)  $, we have
\begin{align*}
\varphi\left(  AXB\right)   &  =\left(  B^{\mathrm{T}}\otimes A\right)
_{\sigma}\varphi\left(  \overline{X}\right) \\
&  =\left(  B^{\mathrm{T}}\otimes A\right)  _{\sigma}\left[
\begin{array}
[c]{cc}%
I_{mn} & 0\\
0 & -I_{mn}%
\end{array}
\right]  \varphi\left(  X\right) \\
&  =\left(  B^{\mathrm{T}}\otimes A\right)  _{\sigma}\left(  I_{mn}\right)
_{\sigma}\varphi\left(  X\right) \\
&  =\left(  \left(  B^{\mathrm{T}}\otimes A\right)  \overline{I}_{mn}\right)
_{\phi}\varphi\left(  X\right) \\
&  =\left(  B^{\mathrm{T}}\otimes A\right)  _{\phi}\varphi\left(  X\right)  .
\end{align*}
The proof is thus finished.
\end{proof}

The following lemma regarding properties of real representation is important
in our further development.

\begin{lemma}
\label{tlm11}Let $A=A_{1}+\mathrm{i}A_{2}\in\mathbf{C}^{m\times m}$ where
$A_{1}$ and $A_{2}$ are real matrices. Then%
\begin{align}
\mathrm{rank}\left(  A_{\sigma}\right)   &  =2\mathrm{rank}\left(  A\right)
,\label{teq51}\\
\mathrm{rank}\left(  I_{2m}-A_{\sigma}\right)   &  =m+\mathrm{rank}\left(
I_{m}-A\overline{A}\right)  . \label{teq52}%
\end{align}

\end{lemma}

\begin{proof}
We first show equation (\ref{teq51}). By using Lemma \ref{tlm2}, we have%
\begin{align*}
\mathrm{rank}\left(  A_{\sigma}\right)   &  =\mathrm{rank}\left(  A_{\sigma
}\left(  I_{m}\right)  _{\sigma}\right)  =\mathrm{rank}\left(  A_{\phi}\right)
\\
&  =\mathrm{rank}\left(  \left[
\begin{array}
[c]{cc}%
\overline{A} & 0\\
0 & A
\end{array}
\right]  \right) \\
&  =\mathrm{rank}\left(  A\right)  +\mathrm{rank}\left(  \overline{A}\right)
\\
&  =2\text{ }\mathrm{rank}\left(  A\right)  .
\end{align*}

We next show (\ref{teq52}). Invoking of (\ref{teq60}) and (\ref{teq65}) we
obtain
\begin{align*}
\mathrm{rank}\left(  I_{2m}-A_{\sigma}\right)   &  =\mathrm{rank}\left(
\left(  I_{m}\right)  _{\sigma}-A_{\sigma}\left(  I_{m}\right)  _{\sigma
}\right) \\
&  =\mathrm{rank}\left(  \left(  I_{m}\right)  _{\sigma}-A_{\phi}\right) \\
&  =\mathrm{rank}\left(  Z_{m}^{\mathrm{H}}\left(  \left(  I_{m}\right)
_{\sigma}-A_{\phi}\right)  Z_{m}\right) \\
&  =\mathrm{rank}\left(  Z_{m}^{\mathrm{H}}\left(  I_{m}\right)  _{\sigma
}Z_{m}-\left[
\begin{array}
[c]{cc}%
\overline{A} & 0\\
0 & A
\end{array}
\right]  \right) \\
&  =\mathrm{rank}\left(  \left[
\begin{array}
[c]{cc}%
-\overline{A} & -\mathrm{i}I_{m}\\
\mathrm{i}I_{m} & -A
\end{array}
\right]  \right) \\
&  =\mathrm{rank}\left(  \left[
\begin{array}
[c]{cc}%
-\overline{A} & -\mathrm{i}I_{m}\\
\mathrm{i}I_{m} & -A
\end{array}
\right]  \left[
\begin{array}
[c]{cc}%
0 & -\mathrm{i}I_{m}\\
I_{m} & \overline{A}%
\end{array}
\right]  \right) \\
&  =\mathrm{rank}\left(  \left[
\begin{array}
[c]{cc}%
-\mathrm{i}I_{m} & 0\\
-A & I_{m}-A\overline{A}%
\end{array}
\right]  \right) \\
&  =m+\mathrm{rank}\left(  I_{m}-A\overline{A}\right)  .
\end{align*}
The proof is finished.
\end{proof}

\subsection{Solvability of Equation $X=A\overline{X}B+C$}

Based on linear equation theory, to study the inhomogeneous matrix equation
(\ref{teq7}), we should study the homogeneous matrix equation $X=A\overline
{X}B.$ Associated with this equation, we define a set $\mathscr{S}^{\prime
}\subset\mathbf{C}^{m\times n}$ as
\[
\mathscr{S}^{\prime}=\left\{  X\in\mathbf{C}^{m\times n}:X=A\overline
{X}B\right\}  .
\]
For any $X_{1}\in\mathscr{S}^{\prime}$ and $X_{2}\in\mathscr{S}^{\prime}$ and
any $\alpha\in\mathbf{C}$, we have%
\[
\left(  X_{1}+\alpha X_{2}\right)  -A\left(  \overline{X_{1}+\alpha X_{2}%
}\right)  B=\alpha X_{2}-\overline{\alpha}A\overline{X_{2}}B,
\]
which is not zero in general, namely, $X_{1}+\alpha X_{2}\notin
\mathscr{S}^{\prime}$ in general. Therefore, $\mathscr{S}^{\prime}$ is
\textit{not} a linear subspace contained in $\mathbf{C}^{m\times n}$ over the
field $\mathbf{C}$. Alternatively, if we define a real set $\mathscr{S}\subset
\mathbf{R}^{mn}$ as
\[
\mathscr{S}=\left\{  \varphi\left(  X\right)  :X=A\overline{X}B,X\in
\mathbf{C}^{m\times n}\right\}  .
\]
then by using the fact that mapping $\varphi\left(  \cdot\right)  $ is
bijective and (\ref{teq77}), we have%
\begin{align}
\mathscr{S}  &  =\left\{  \varphi\left(  X\right)  :\varphi\left(  X\right)
=\varphi\left(  A\overline{X}B\right)  ,\text{ }X\in\mathbf{C}^{m\times
n}\right\} \nonumber\\
&  =\left\{  \varphi\left(  X\right)  :\varphi\left(  X\right)  =\left(
B^{\mathrm{T}}\otimes A\right)  _{\sigma}\varphi\left(  X\right)  ,\text{
}X\in\mathbf{C}^{m\times n}\right\} \nonumber\\
&  =\left\{  x:x=\left(  B^{\mathrm{T}}\otimes A\right)  _{\sigma}x,\text{
}x\in\mathbf{R}^{2mn}\right\} \nonumber\\
&  =\left\{  x:\left(  I_{2mn}-\left(  B^{\mathrm{T}}\otimes A\right)
_{\sigma}\right)  x=0,\text{ }x\in\mathbf{R}^{2mn}\right\}  , \label{teqs}%
\end{align}
which is clearly a linear subspace contained in $\mathbf{R}^{2mn}$ over the
field $\mathbf{R}.$ Moreover, we can see from (\ref{teqs}) that%
\begin{equation}
\dim\left(  \mathscr{S}\right)  =2mn-\mathrm{rank}\left(  I_{2mn}-\left(
B^{\mathrm{T}}\otimes A\right)  _{\sigma}\right)  . \label{teq79}%
\end{equation}
where $\dim\left(  \mathscr{S}\right)  $ denotes the degree of freedom
(measured by the maximal number of real parameters) in the solution.
Therefore, $\mathscr{S}$ is more suitable for characterizing the general
solutions of equation (\ref{teq7}). In fact, if we let
\[
\mathscr{X}=\left\{  X\in\mathbf{C}^{m\times n}:X=A\overline{X}B+C\right\}  ,
\]
and $X^{\dag}$ be any particular solution of equation (\ref{teq7}), then%
\[
\mathscr{X}=\left\{  X^{\dag}+\varphi^{-1}\left(  x\right)  :\forall
x\in\mathscr{S}\right\}  .
\]

To present our main result, we need an auxiliary result.

\begin{lemma}
\label{tlm14}Equation (\ref{teq7}) has a unique solution for arbitrary $C$ if
and only if
\begin{equation}
\eta\gamma\neq1,\text{ \ }\forall\eta\in\lambda\left\{  A\overline{A}\right\}
,\forall\gamma\in\lambda\left\{  \overline{B}B\right\}  . \label{teq19}%
\end{equation}

\end{lemma}

\begin{proof}
Taking $\varphi\left(  \cdot\right)  $ on both sides of (\ref{teq7}) and using
(\ref{teq77}) give%
\begin{equation}
\left(  I_{2mn}-\left(  B^{\mathrm{T}}\otimes A\right)  _{\sigma}\right)
\varphi\left(  X\right)  =\varphi\left(  C\right)  . \label{teq80}%
\end{equation}

Therefore, equation (\ref{teq7}) has a unique solution for arbitrary $C$ if
and only if%
\begin{equation}
\mathrm{rank}\left(  I_{2mn}-\left(  B^{\mathrm{T}}\otimes A\right)  _{\sigma
}\right)  =2mn. \label{teq61}%
\end{equation}
On the other hand, in view of Lemma \ref{tlm11}, we can see that%
\begin{align}
\mathrm{rank}\left(  I_{2mn}-\left(  B^{\mathrm{T}}\otimes A\right)  _{\sigma
}\right)   &  =mn+\mathrm{rank}\left(  I_{mn}-\left(  B^{\mathrm{T}}\otimes
A\right)  \left(  \overline{B^{\mathrm{T}}\otimes A}\right)  \right)
\nonumber\\
&  =mn+\mathrm{rank}\left(  I_{mn}-\left(  \overline{B}B\right)  ^{\mathrm{T}%
}\otimes A\overline{A}\right)  . \label{teq49}%
\end{align}
Hence, (\ref{teq61}) is true if and only if%
\[
\mathrm{rank}\left(  I_{mn}-\left(  \overline{B}B\right)  ^{\mathrm{T}}\otimes
A\overline{A}\right)  =mn,
\]
or, equivalent, $1\notin\lambda\left\{  \left(  \overline{B}B\right)
^{\mathrm{T}}\otimes A\overline{A}\right\}  .$ This is further equivalent to
(\ref{teq19}). The proof is completed.
\end{proof}

We then can give the following alternative conditions regarding solvability of
equation (\ref{teq7}), which can be regarded as the generalization of Theorem
\ref{tth1} to equation (\ref{teq7}).

\begin{theorem}
\label{tth3}Matrix equation (\ref{teq7}) is solvable if and only if the
following equation%
\begin{equation}
W=A\overline{A}W\overline{B}B+A\overline{C}B+C, \label{teq9}%
\end{equation}
is solvable with unknown $W$. More specifically:

\begin{enumerate}
\item If $X$ is a solution of equation (\ref{teq7}), then $W=X$ is a solution
of equation (\ref{teq9}).

\item If $\boldsymbol{W}$ is the general closed-form solution of equation
(\ref{teq9}), then the general closed-form solution of equation (\ref{teq7})
is given by
\begin{equation}
\boldsymbol{X}=\frac{1}{2}\left(  \boldsymbol{W}+A\overline{\boldsymbol{W}%
}B+C\right)  . \label{teq76}%
\end{equation}

\item Equation (\ref{teq7}) has a unique solution $X^{\star}$ if and only if
equation (\ref{teq9}) has a unique solution $W^{\star}$. Moreover, $X^{\star
}=W^{\star}.$
\end{enumerate}
\end{theorem}

\begin{proof}
\textit{Proofs of Item 1 and Item 2}. Taking $\varphi\left(  \cdot\right)  $
on both sides of (\ref{teq9}), using Lemma \ref{tlm12} and denoting
$E=B^{\mathrm{T}}\otimes A,$ we get%
\begin{align*}
\varphi\left(  W\right)   &  =\varphi\left(  A\overline{A}W\overline
{B}B\right)  +\varphi\left(  A\overline{C}B\right)  +\varphi\left(  C\right)
\\
&  =E_{\sigma}\varphi\left(  A\overline{W}B\right)  +\left(  I_{2mn}%
+E_{\sigma}\right)  \varphi\left(  C\right) \\
&  =E_{\sigma}^{2}\varphi\left(  W\right)  +\left(  I_{2mn}+E_{\sigma}\right)
\varphi\left(  C\right)  ,
\end{align*}
which is further equivalent to%
\begin{equation}
\left(  I_{2mn}-E_{\sigma}^{2}\right)  \varphi\left(  W\right)  =\left(
I_{2mn}+E_{\sigma}\right)  \varphi\left(  C\right)  . \label{teq81}%
\end{equation}

Similarly, equation (\ref{teq7}) is equivalent to (\ref{teq80}) which can be
rewritten as%
\begin{equation}
\left(  I_{2mn}-E_{\sigma}\right)  \varphi\left(  X\right)  =\varphi\left(
C\right)  , \label{teq80a}%
\end{equation}
and the expression in (\ref{teq76}) is equivalent to
\begin{equation}
\varphi\left(  \boldsymbol{X}\right)  =\frac{1}{2}\left(  I_{2mn}+E_{\sigma
}\right)  \varphi\left(  \boldsymbol{W}\right)  +\frac{1}{2}\varphi\left(
C\right)  . \label{teq95}%
\end{equation}
Since equation (\ref{teq81}) is in the form of (\ref{teq90}), equation
(\ref{teq80a}) is in the form of (\ref{teq92}) and expression (\ref{teq95}) is
in the form of (\ref{teq91}), Item 1 and Item 2 follows from Proposition
\ref{tpp1} directly.

\textit{Proof of Item 3}. Notice that, by using Lemma \ref{tlm14}, equation
(\ref{teq7}) has a unique solution $X^{\star}$ if and only if (\ref{teq19}) is
satisfied, which is also the necessary and sufficient condition for the
existence of unique solution of equation (\ref{teq9}). The proof is finished.
\end{proof}

\begin{remark}
The difference between Theorem \ref{tth1} and Theorem \ref{tth3} is that
condition (\ref{teq19}) in Theorem \ref{tth3} is necessary and sufficient for
the existence of unique solution of equation (\ref{teq7}).
\end{remark}

\begin{remark}
Different from Proposition \ref{tlm4}, the solvability condition for equation
(\ref{teq7}) in Theorem \ref{tth3} is related with solvability of equation
(\ref{teq9}) which has the same dimension as (\ref{teq7}). Also, different
from Proposition \ref{tlm4}, we have given in Theorem \ref{tth3} a necessary
and sufficient condition for the existence of unique solution of equation
(\ref{teq7}).
\end{remark}

The following corollary is then a consequence of Theorem \ref{tth3}.

\begin{corollary}
\label{tcoro1}Equation (\ref{teq7}) is solvable if and only if%
\begin{equation}
\mathrm{rank}\left(  I_{mn}-\left(  \overline{B}B\right)  ^{\mathrm{T}}\otimes
A\overline{A}\right)  =\mathrm{rank}\left(  \left[
\begin{array}
[c]{cc}%
I_{mn}-\left(  \overline{B}B\right)  ^{\mathrm{T}}\otimes A\overline{A} &
\left(  B^{\mathrm{T}}\otimes A\right)  \mathrm{vec}\left(  \overline
{C}\right)  +\mathrm{vec}\left(  C\right)
\end{array}
\right]  \right)  . \label{teq55}%
\end{equation}
If the above relation is satisfied, then the degree of freedom (measured by
the maximal number of free real parameters), namely, $\dim\left(
\mathscr{S}\right)  ,$ in the solution is given by%
\begin{equation}
\dim\left(  \mathscr{S}\right)  =mn-\mathrm{rank}\left(  I_{mn}-\left(
\overline{B}B\right)  ^{\mathrm{T}}\otimes A\overline{A}\right)  .
\label{teq57}%
\end{equation}

\end{corollary}

\subsection{Relationship Between Equations $X=A\overline{X}B+C$ and
$Y=A_{\sigma}YB_{\sigma}+C_{\sigma}$}

From Lemma \ref{tlm4} we can see easily that if equation (\ref{teq14}) has a
unique solution, then equation (\ref{teq7}) must also have a unique solution.
But the converse is not clear. In this subsection, by using the results
obtained in the above two subsections, we are able to close this gap. Our
result reveals some deep relationships between these two equations
(\ref{teq7}) and (\ref{teq14}).

\begin{theorem}
\label{tth9} Equation (\ref{teq7}) has a unique solution $X^{\star}$ if and
only if equation (\ref{teq14}) has a unique solution $Y^{\star}$. Moreover,
there holds $Y^{\star}=X_{\sigma}^{\star}$.
\end{theorem}

\begin{proof}
Obviously, equation (\ref{teq14}) has a unique solution for arbitrary
$C_{\sigma}$ if and only if%
\begin{equation}
\alpha\beta\neq1,\text{ \ \ }\forall\alpha\in\lambda\left\{  A_{\sigma
}\right\}  ,\forall\beta\in\lambda\left\{  B_{\sigma}\right\}  . \label{teq18}%
\end{equation}

For an arbitrary matrix $A=A_{1}+\mathrm{i}A_{2}\in\mathbf{C}^{m\times m}$
where $A_{1},A_{2}\in\mathbf{R}^{m\times m},$ we can compute
\begin{align}
\det\left(  sI_{2m}-A_{\sigma}\right)   &  =\mathrm{i}^{m}\det\left(  \left[
\begin{array}
[c]{cc}%
I_{m} & 0\\
-I_{m} & \mathrm{i}I_{m}%
\end{array}
\right]  \right)  \det\left(  \left[
\begin{array}
[c]{cc}%
sI_{m}-A_{1} & -A_{2}\\
-A_{2} & sI_{m}+A_{1}%
\end{array}
\right]  \right)  \det\left(  \left[
\begin{array}
[c]{cc}%
\frac{\mathrm{i}}{2}I_{m} & I_{m}\\
\frac{1}{2}I_{m} & \mathrm{i}I_{m}%
\end{array}
\right]  \right) \nonumber\\
&  =\mathrm{i}^{m}\det\left(  \left[
\begin{array}
[c]{cc}%
I_{m} & 0\\
-I_{m} & \mathrm{i}I_{m}%
\end{array}
\right]  \left[
\begin{array}
[c]{cc}%
sI_{m}-A_{1} & -A_{2}\\
-A_{2} & sI_{m}+A_{1}%
\end{array}
\right]  \left[
\begin{array}
[c]{cc}%
\frac{\mathrm{i}}{2}I_{m} & I_{m}\\
\frac{1}{2}I_{m} & \mathrm{i}I_{m}%
\end{array}
\right]  \right) \nonumber\\
&  =\mathrm{i}^{m}\det\left(  \left[
\begin{array}
[c]{cc}%
\frac{\mathrm{i}}{2}sI_{m}-\frac{\mathrm{i}}{2}\overline{A} & sI_{m}-A\\
\mathrm{i}\overline{A} & -2sI_{m}%
\end{array}
\right]  \right) \nonumber\\
&  =\mathrm{i}^{m}\det\left(  -2sI_{m}\right)  \det\left(  \frac{\mathrm{i}%
}{2}sI_{m}-\frac{\mathrm{i}}{2}\overline{A}-\left(  sI_{m}-A\right)  \left(
-2sI_{m}\right)  ^{-1}\mathrm{i}\overline{A}\right) \nonumber\\
&  =\mathrm{i}^{m}\det\left(  \left(  -2sI_{m}\right)  \left(  \frac
{\mathrm{i}}{2}sI_{m}-\frac{\mathrm{i}}{2}\overline{A}\right)  -\left(
sI_{m}-A\right)  \mathrm{i}\overline{A}\right) \nonumber\\
&  =\mathrm{i}^{m}\det\left(  -\mathrm{i}\left(  s^{2}I_{m}-A\overline
{A}\right)  \right) \nonumber\\
&  =\det\left(  s^{2}I_{m}-A\overline{A}\right) \nonumber\\
&  =\det\left(  s^{2}I_{m}-\overline{A}A\right)  . \label{teq20}%
\end{align}
where we have used Lemma \ref{tlm10}.

Let $\alpha$ and $\beta$ be two arbitrary elements in $\lambda\left\{
A_{\sigma}\right\}  $ and $\lambda\left\{  B_{\sigma}\right\}  ,$
respectively. Then it follows from Lemma \ref{tlm9} that $\left\{  \pm
\alpha,\pm\overline{\alpha}\right\}  \subset\lambda\left\{  A_{\sigma
}\right\}  ,\left\{  \pm\beta,\pm\overline{\beta}\right\}  \subset
\lambda\left\{  B_{\sigma}\right\}  $ and it follows from (\ref{teq20}) that
$\left\{  \alpha^{2},\overline{\alpha}^{2}\right\}  \subset\lambda\left\{
A\overline{A}\right\}  ,\{\beta^{2},\overline{\beta}^{2}\}\subset
\lambda\left\{  \overline{B}B\right\}  .$ Direct manipulation shows that
\[
st\neq1,\forall s\in\left\{  \pm\alpha,\pm\overline{\alpha}\right\}
\subset\lambda\left\{  A_{\sigma}\right\}  ,\forall t\in\left\{  \pm\beta
,\pm\overline{\beta}\right\}  \subset\lambda\left\{  B_{\sigma}\right\}  ,
\]
if and only if $\alpha\beta\neq\pm1$ and $\alpha\overline{\beta}\neq\pm1.$
Also, direct computation indicates that%
\[
st\neq1,\forall s\in\left\{  \alpha^{2},\overline{\alpha}^{2}\right\}
\subset\lambda\left\{  A\overline{A}\right\}  ,\forall t\in\{\beta
^{2},\overline{\beta}^{2}\}\subset\lambda\left\{  \overline{B}B\right\}  ,
\]
if and only if $\alpha\beta\neq\pm1$ and $\alpha\overline{\beta}\neq\pm1.$
According to the arbitrariness of $\alpha$ and $\beta,$ we conclude from the
above that (\ref{teq19}) and (\ref{teq18}) are equivalent.

Finally, it is trivial to show that $Y^{\star}=X_{\sigma}^{\star}.$ The proof
is completed.
\end{proof}

Proposition \ref{tlm4} and Theorem \ref{tth9} are still unable to answer the
following question: If $\boldsymbol{Y}$ is the general real solution of
equation (\ref{teq14}), then is the following
\begin{equation}
\boldsymbol{X}^{\prime}\triangleq\frac{1}{4}\left[
\begin{array}
[c]{cc}%
I_{m} & \mathrm{i}I_{m}%
\end{array}
\right]  \left(  \boldsymbol{Y}+Q_{m}\boldsymbol{Y}Q_{n}\right)  \left[
\begin{array}
[c]{c}%
I_{n}\\
\mathrm{i}I_{n}%
\end{array}
\right]  \label{teq25}%
\end{equation}
the general solution of equation (\ref{teq7})? The following example may
indicate that the answer is positive.

\begin{example}
\label{texamp2}Consider a linear equation in the form of (\ref{teq7}) with%
\[
A=\left[
\begin{array}
[c]{cc}%
2 & 0\\
0 & \mathrm{i}%
\end{array}
\right]  ,\text{ }B=1,\text{ }C=\left[
\begin{array}
[c]{c}%
C_{1}\\
C_{2}%
\end{array}
\right]  =\left[
\begin{array}
[c]{c}%
C_{11}+\mathrm{i}C_{12}\\
C_{21}+\mathrm{i}C_{22}%
\end{array}
\right]  ,
\]
where $C_{ij},i,j=1,2$ are real scalars. Then equation (\ref{teq9}) can be
written as%
\[
W=\left[
\begin{array}
[c]{c}%
w_{1}\\
w_{2}%
\end{array}
\right]  =A\overline{A}W\overline{B}B+A\overline{C}B+C=\left[
\begin{array}
[c]{c}%
4w_{1}+2\overline{C_{1}}+C_{1}\\
w_{2}+\mathrm{i}\overline{C_{2}}+C_{2}%
\end{array}
\right]  ,
\]
which is solvable if and only if $\mathrm{i}\overline{C_{2}}+C_{2}=0$.
Consequently, the general solution of equation (\ref{teq9}) is%
\[
\boldsymbol{W}=\left[
\begin{array}
[c]{c}%
-\frac{1}{3}\left(  2\overline{C_{1}}+C_{1}\right) \\
t
\end{array}
\right]  ,\text{ }\forall t=t_{1}+\mathrm{i}t_{2}\in\mathbf{C,}%
\]
where $t_{1}$ and $t_{2}$ are real scalars. Hence, according to Theorem
\ref{tth3}, the general solution of equation (\ref{teq7}) is given by%
\[
\boldsymbol{X}=\frac{1}{2}\left(  \boldsymbol{W}+A\overline{\boldsymbol{W}%
}B+C\right)  =\left[
\begin{array}
[c]{c}%
-\frac{1}{3}\left(  C_{1}+2\overline{C_{1}}\right) \\
\frac{1}{2}\left(  1+\mathrm{i}\right)  \left(  t_{1}+t_{2}\right)  +\frac
{1}{2}C_{2}%
\end{array}
\right]  .
\]
On the other hand, direct computation shows that the general real solution of
equation (\ref{teq14}) is%
\[
\boldsymbol{Y}=\left[
\begin{array}
[c]{cc}%
0 & 0\\
\frac{1}{2}s & \frac{1}{2}s\\
0 & 0\\
\frac{1}{2}s & -\frac{1}{2}s
\end{array}
\right]  +\left[
\begin{array}
[c]{cc}%
-C_{11} & \frac{1}{3}C_{12}\\
\frac{1}{2}C_{21} & \frac{1}{2}C_{22}\\
\frac{1}{3}C_{12} & C_{11}\\
\frac{1}{2}C_{22} & -\frac{1}{2}C_{21}%
\end{array}
\right]  ,\forall s\in\mathbf{R}.
\]
Then $\boldsymbol{X}^{\prime}$ defined in (\ref{teq25}) can be computed as%
\[
\boldsymbol{X}^{\prime}=\left[
\begin{array}
[c]{c}%
-\frac{1}{3}\left(  C_{1}+2\overline{C_{1}}\right) \\
\frac{1}{2}C_{2}%
\end{array}
\right]  +\left[
\begin{array}
[c]{c}%
0\\
\frac{1}{2}\left(  1+\mathrm{i}\right)  s
\end{array}
\right]  .
\]
It is clear that $\boldsymbol{X}^{\prime}=\boldsymbol{X}$ if we set
$t_{1}+t_{2}=s$ in $\boldsymbol{X}.$
\end{example}

However, at present we cannot prove the correctness the above result in
general. This would be a conjecture that needs further study.

\begin{conjecture}
If $\boldsymbol{Y}$ is the general real solution of equation (\ref{teq7}),
then the general solution of equation (\ref{teq7}) is $\boldsymbol{X}^{\prime
}$ defined in (\ref{teq25}).
\end{conjecture}

We should point out that even though the above conjecture is true, compared
with the results in Theorem \ref{tth3}, using solutions of equation
(\ref{teq14}) to construct solution of equation (\ref{teq7}) has no advantage
since the dimensions of equation (\ref{teq14}) has doubled the dimensions of
equation (\ref{teq9}) used in Theorem \ref{tth3}.

\subsection{Closed-Form and Iterative Solutions of Equation $X=A\overline
{X}B+C$}

By combining Lemma \ref{tlm3} and Theorem \ref{tth3} we can derive the
following corollary regarding closed-from solutions of equation (\ref{teq7}).
We notice that this result is exactly Theorem 4.4\ in \cite{jw03laa} where the
real representation method is developed to derive such result. In contrast, in
this paper such result is only a consequence of Theorem \ref{tth3} and Lemma
\ref{tlm3} without using the intricate properties of the real representation.

\begin{corollary}
\label{tcoro2}Let $A\in\mathbf{C}^{m\times m},B\in\mathbf{C}^{n\times n}%
,C\in\mathbf{C}^{m\times n}$ be given and%
\[
h_{A\overline{A}}\left(  s\right)  =\det\left(  I_{m}-sA\overline{A}\right)
=\alpha_{m}+\sum\limits_{i=1}^{m}\alpha_{m-i}s^{i},\text{ }\alpha_{m}=1.
\]

\begin{enumerate}
\item If equation (\ref{teq7}) has a solution $X,$ then%
\[
Xh_{A\overline{A}}\left(  \overline{B}B\right)  =\sum\limits_{k=1}^{m}%
\sum\limits_{s=1}^{k}\alpha_{k}\left(  A\overline{A}\right)  ^{k-s}\left(
A\overline{C}B+C\right)  \left(  \overline{B}B\right)  ^{m-s}.
\]

\item If equation (\ref{teq7}) has a unique solution $X^{\star},$ then%
\[
X^{\star}=\left(  \sum\limits_{k=1}^{m}\sum\limits_{s=1}^{k}\alpha_{k}\left(
A\overline{A}\right)  ^{k-s}\left(  A\overline{C}B+C\right)  \left(
\overline{B}B\right)  ^{m-s}\right)  \left(  h_{A\overline{A}}\left(
\overline{B}B\right)  \right)  ^{-1}.
\]

\end{enumerate}
\end{corollary}

Similar to the Smith Iteration for equation (\ref{teq}), we can also construct
the Smith iteration for equation (\ref{teq7}) as follows:%
\begin{equation}
X_{k+1}=A\overline{X_{k}}B+C,\text{ \ }\forall X_{0}\in\mathbf{C}^{m\times n}.
\label{teq40}%
\end{equation}
The following theorem is concerned with the convergence of this iteration.

\begin{theorem}
\label{tth8}The Smith iteration (\ref{teq40}) converges to the unique solution
$X^{\star}$ of equation (\ref{teq7}) for arbitrary initial condition $X_{0}$
if and only if%
\begin{equation}
\rho\left(  A\overline{A}\right)  \rho\left(  \overline{B}B\right)  <1.
\label{teq43}%
\end{equation}
Moreover, the asymptotic exponential convergence rate is $-\ln\left(
\rho\left(  A\overline{A}\right)  \rho\left(  \overline{B}B\right)  \right)
.$
\end{theorem}

\begin{proof}
We first show that if the iteration (\ref{teq40}) converges, then
(\ref{teq43}) should be satisfied. As $\left\{  X_{k}\right\}  _{k=0}^{\infty
}$ converges, we know that the subsequence $\left\{  X_{2i}\right\}
_{i=0}^{\infty}$ also converges. It is easy to verify that $\left\{
X_{2i}\right\}  _{i=0}^{\infty}$ is generated by the following iteration%
\[
X_{2\left(  i+1\right)  }=A\overline{A}X_{2i}\overline{B}B+A\overline
{C}B+C,\text{ \ }\forall X_{0}\in\mathbf{C}^{m\times n}.
\]
Notice that the above iteration is a standard Smith iteration which converges
to a constant matrix independent of the initial condition if and only if
(\ref{teq43}) is true.

We next show that if (\ref{teq43}) is satisfied, then the iteration
(\ref{teq40}) must converge. Notice that the iteration (\ref{teq40}) can be
equivalently rewritten as%
\begin{equation}
\left(  X_{k+1}\right)  _{\sigma}=A_{\sigma}\left(  X_{k}\right)  _{\sigma
}B_{\sigma}+C_{\sigma},\text{ \ }\forall\left(  X_{0}\right)  _{\sigma}%
\in\mathbf{R}^{2m\times2n}. \label{teq44}%
\end{equation}
From (\ref{teq20}) we can see that%
\[
\rho\left(  A_{\sigma}\right)  \rho\left(  B_{\sigma}\right)  =\sqrt
{\rho\left(  \overline{A}A\right)  }\sqrt{\rho\left(  \overline{B}B\right)
}<1,
\]
which indicates that (\ref{teq44}) converges to a constant matrix $\left(
X_{\infty}\right)  _{\sigma}$ independent of initial condition, namely,
$\left(  X_{\infty}\right)  _{\sigma}=A_{\sigma}\left(  X_{\infty}\right)
_{\sigma}B_{\sigma}+C_{\sigma},$ or equivalently, $X_{\infty}=A\overline
{X_{\infty}}B+C$. The proof is completed by observing that equation
(\ref{teq7}) has a unique solution under the condition (\ref{teq43}).
\end{proof}

\begin{remark}
\label{trm3}Under condition (\ref{teq43}), it follows from Theorem \ref{tth3}
that the unique solution $X^{\star}$ to equation (\ref{teq7}) is also the
unique solution of equation (\ref{teq9}), namely,%
\begin{equation}
X^{\star}=\sum\limits_{i=0}^{\infty}\left(  A\overline{A}\right)  ^{i}\left(
A\overline{C}B+C\right)  \left(  \overline{B}B\right)  ^{i}. \label{teq17}%
\end{equation}
Therefore, similar to Remark \ref{trm1} and Remark \ref{trm2}, we can also
construct Smith $\left(  l\right)  $ iteration and $r$-Smith iteration for
solving equation (\ref{teq7}). The details are omitted due to limited space.
\end{remark}

\section{\label{tsec4}Equation $X=AX^{\mathrm{H}}B+C$}

In this section, we study the following matrix equation:
\begin{equation}
X=AX^{\mathrm{H}}B+C, \label{teq3}%
\end{equation}
where $A\in\mathbf{C}^{m\times n},B\in\mathbf{C}^{m\times n}$ and
$C\in\mathbf{C}^{m\times n}$ are given, and $X\in\mathbf{C}^{m\times n}$ is
unknown. In this section, we will generalize the results for equations
(\ref{teq}) and (\ref{teq7}) to this equation.

\subsection{\label{tsec4.1}Solvability of Equation $X=AX^{\mathrm{H}}B+C$}

We first show a result parallel to Lemma \ref{tlm14}.

\begin{lemma}
\label{tlm20}Equation (\ref{teq3}) has a unique solution for arbitrary $C$ if
and only if
\begin{equation}
\overline{\eta}\gamma\neq1,\text{ \ }\forall\eta,\gamma\in\lambda\left\{
AB^{\mathrm{H}}\right\}  . \label{teq6}%
\end{equation}

\end{lemma}

\begin{proof}
Notice that, for any complex matrix $X\in\mathbf{C}^{m\times n},$ we have%
\[
\varphi\left(  X^{\mathrm{T}}\right)  =\left[
\begin{array}
[c]{c}%
\mathrm{vec}\left(  X_{1}^{\mathrm{T}}\right) \\
\mathrm{vec}\left(  X_{2}^{\mathrm{T}}\right)
\end{array}
\right]  =\left[
\begin{array}
[c]{c}%
P_{\left(  m,n\right)  }\mathrm{vec}\left(  X_{1}\right) \\
P_{\left(  m,n\right)  }\mathrm{vec}\left(  X_{2}\right)
\end{array}
\right]  =\mathcal{P}_{\left(  m,n\right)  }\varphi\left(  X\right)  ,
\]
where%
\begin{equation}
\mathcal{P}_{\left(  m,n\right)  }=\left[
\begin{array}
[c]{cc}%
P_{\left(  m,n\right)  } & 0\\
0 & P_{\left(  m,n\right)  }%
\end{array}
\right]  . \label{teq36}%
\end{equation}
With this, by taking $\varphi\left(  \cdot\right)  $ on both sides of
(\ref{teq3}), using Lemma \ref{tlm12}, and denoting $E=B^{\mathrm{T}}\otimes
A,$ we have%
\begin{align*}
\varphi\left(  X\right)   &  =\varphi\left(  AX^{\mathrm{H}}B\right)
+\varphi\left(  C\right) \\
&  =\left(  B^{\mathrm{T}}\otimes A\right)  _{\sigma}\varphi\left(
X^{\mathrm{T}}\right)  +\varphi\left(  C\right) \\
&  =E_{\sigma}\mathcal{P}_{\left(  m,n\right)  }\varphi\left(  X\right)
+\varphi\left(  C\right)  .
\end{align*}
Therefore, equation (\ref{teq3}) has a unique solution for arbitrary $C$ if
and only if%
\begin{equation}
\mathrm{rank}\left(  I_{2mn}-E_{\sigma}\mathcal{P}_{\left(  m,n\right)
}\right)  =2mn. \label{teq33}%
\end{equation}

Similar to the proof of relation (\ref{teq52}), by using Lemma \ref{tlm2}, we
can compute%
\begin{align}
\mathrm{rank}\left(  I_{2mn}-E_{\sigma}\mathcal{P}_{\left(  m,n\right)
}\right)   &  =\mathrm{rank}\left(  \mathcal{P}_{\left(  n,m\right)
}-E_{\sigma}\right) \nonumber\\
&  =\mathrm{rank}\left(  \mathcal{P}_{\left(  n,m\right)  }\left(
I_{mn}\right)  _{\sigma}-E_{\sigma}\left(  I_{mn}\right)  _{\sigma}\right)
\nonumber\\
&  =\mathrm{rank}\left(  \mathcal{P}_{\left(  n,m\right)  }\left(
I_{mn}\right)  _{\sigma}-E_{\phi}\right) \nonumber\\
&  =\mathrm{rank}\left(  Z_{mn}^{\mathrm{H}}\left(  \mathcal{P}_{\left(
n,m\right)  }\left(  I_{mn}\right)  _{\sigma}-E_{\phi}\right)  Z_{mn}\right)
,\nonumber\\
&  =\mathrm{rank}\left(  Z_{mn}^{\mathrm{H}}\mathcal{P}_{\left(  n,m\right)
}\left(  I_{mn}\right)  _{\sigma}Z_{mn}-\left[
\begin{array}
[c]{cc}%
\overline{E} & 0\\
0 & E
\end{array}
\right]  \right) \nonumber\\
&  =\mathrm{rank}\left(  \Omega-\left[
\begin{array}
[c]{cc}%
P_{\left(  m,n\right)  }\overline{E}P_{\left(  m,n\right)  } & 0\\
0 & E
\end{array}
\right]  \right)  , \label{teq35}%
\end{align}
where%
\begin{align*}
\Omega &  \triangleq\left[
\begin{array}
[c]{cc}%
P_{\left(  m,n\right)  } & 0\\
0 & I_{mn}%
\end{array}
\right]  Z_{mn}^{\mathrm{H}}\mathcal{P}_{\left(  n,m\right)  }\left(
I_{mn}\right)  _{\sigma}Z_{mn}\left[
\begin{array}
[c]{cc}%
P_{\left(  m,n\right)  } & 0\\
0 & I_{mn}%
\end{array}
\right] \\
&  =\left[
\begin{array}
[c]{cc}%
0 & -\mathrm{i}I_{mn}\\
\mathrm{i}I_{mn} & 0
\end{array}
\right]  .
\end{align*}
Hence, we get from (\ref{teq35}) that%
\begin{align}
\mathrm{rank}\left(  I_{2mn}-E_{\sigma}\mathcal{P}_{\left(  m,n\right)
}\right)   &  =\mathrm{rank}\left(  \left[
\begin{array}
[c]{cc}%
-P_{\left(  m,n\right)  }\overline{E}P_{\left(  m,n\right)  } & -\mathrm{i}%
I_{mn}\\
\mathrm{i}I_{mn} & -E
\end{array}
\right]  \right) \nonumber\\
&  =\mathrm{rank}\left(  \left[
\begin{array}
[c]{cc}%
-P_{\left(  m,n\right)  }\overline{E}P_{\left(  m,n\right)  } & -\mathrm{i}%
I_{mn}\\
\mathrm{i}I_{mn} & -E
\end{array}
\right]  \left[
\begin{array}
[c]{cc}%
0 & -\mathrm{i}I_{mn}\\
I_{mn} & P_{\left(  m,n\right)  }\overline{E}P_{\left(  m,n\right)  }%
\end{array}
\right]  \right) \nonumber\\
&  =\mathrm{rank}\left(  \left[
\begin{array}
[c]{cc}%
-\mathrm{i}I_{mn} & 0\\
-E & I_{mn}-EP_{\left(  m,n\right)  }\overline{E}P_{\left(  m,n\right)  }%
\end{array}
\right]  \right) \nonumber\\
&  =mn+\mathrm{rank}\left(  I_{mn}-EP_{\left(  m,n\right)  }\overline
{E}P_{\left(  m,n\right)  }\right)  . \label{teq34}%
\end{align}
Notice that%
\begin{align*}
EP_{\left(  m,n\right)  }\overline{E}P_{\left(  m,n\right)  }  &  =EP_{\left(
m,n\right)  }\overline{\left(  B^{\mathrm{T}}\otimes A\right)  }P_{\left(
m,n\right)  }\\
&  =EP_{\left(  m,n\right)  }\left(  B^{\mathrm{H}}\otimes\overline{A}\right)
P_{\left(  m,n\right)  }\\
&  =\left(  B^{\mathrm{T}}\otimes A\right)  \left(  \overline{A}\otimes
B^{\mathrm{H}}\right) \\
&  =B^{\mathrm{T}}\overline{A}\otimes AB^{\mathrm{H}}\\
&  =\left(  A^{\mathrm{H}}B\right)  ^{\mathrm{T}}\otimes AB^{\mathrm{H}}.
\end{align*}
Hence it follows from (\ref{teq34}) that condition (\ref{teq33}) is true if
and only if
\begin{equation}
1\notin\lambda\left\{  \left(  A^{\mathrm{H}}B\right)  ^{\mathrm{T}}\otimes
AB^{\mathrm{H}}\right\}  \label{teq5a}%
\end{equation}
which is equivalent to (\ref{teq6}) because%
\begin{align*}
\lambda\left\{  \left(  A^{\mathrm{H}}B\right)  ^{\mathrm{T}}\otimes
AB^{\mathrm{H}}\right\}   &  =\lambda\left\{  B^{\mathrm{T}}\overline
{A}\otimes AB^{\mathrm{H}}\right\} \\
&  =\lambda\left\{  \overline{AB^{\mathrm{H}}}\otimes AB^{\mathrm{H}}\right\}
\\
&  =\left\{  \overline{\eta}\gamma,\text{ \ }\forall\eta,\gamma\in
\lambda\left\{  AB^{\mathrm{H}}\right\}  \right\}  .
\end{align*}
The proof is completed.
\end{proof}

We are now able to present the results regarding the solvability of equation
(\ref{teq3}).

\begin{theorem}
\label{tth2}Matrix equation (\ref{teq3}) is solvable if and only if the
following equation%
\begin{equation}
W=AB^{\mathrm{H}}WA^{\mathrm{H}}B+AC^{\mathrm{H}}B+C, \label{teq4}%
\end{equation}
is solvable with unknown $W.$ More specifically,

\begin{enumerate}
\item If $X$ is a solution of equation (\ref{teq3}), then $W=X$ is also a
solution of equation (\ref{teq4}).

\item If $\boldsymbol{W}$ is the general solution of equation (\ref{teq4}),
then the general solution to equation (\ref{teq3}) is%
\begin{equation}
\boldsymbol{X}=\frac{1}{2}\left(  \boldsymbol{W}+A\boldsymbol{W}^{\mathrm{H}%
}B\right)  +\frac{1}{2}C. \label{teq105}%
\end{equation}

\item Equation (\ref{teq3}) has a unique solution $X^{\star}$ if and only if
equation (\ref{teq4}) has a unique solution $W^{\star}$, namely, (\ref{teq6})
is satisfied. Moreover, $X^{\star}=W^{\star}.$
\end{enumerate}
\end{theorem}

\begin{proof}
\textit{Proofs of Item 1 and Item 2}. For any matrices $A\in\mathbf{C}%
^{m\times n},P\in\mathbf{R}^{p\times m}$ and $Q\in\mathbf{R}^{n\times q}$, it
cam be readily verified that%
\[
\left(  PAQ\right)  _{\sigma}=\left[
\begin{array}
[c]{cc}%
P & 0\\
0 & P
\end{array}
\right]  A_{\sigma}\left[
\begin{array}
[c]{cc}%
Q & 0\\
0 & Q
\end{array}
\right]  .
\]
By using this fact, Lemma \ref{tlm0} and Lemma \ref{tlm12}, we obtain
\begin{align*}
\varphi\left(  W\right)   &  =\varphi\left(  A\overline{B}^{\mathrm{T}%
}W\overline{A}^{\mathrm{T}}B\right)  +\varphi\left(  A\overline{C^{\mathrm{T}%
}}B\right)  +\varphi\left(  C\right) \\
&  =E_{\sigma}\varphi\left(  B^{\mathrm{T}}\overline{W}A^{\mathrm{T}}\right)
+E_{\sigma}\varphi\left(  C^{\mathrm{T}}\right)  +\varphi\left(  C\right) \\
&  =E_{\sigma}\left(  A\otimes B^{\mathrm{T}}\right)  _{\sigma}\varphi\left(
W\right)  +\left(  I_{2mn}+E_{\sigma}\mathcal{P}_{\left(  m,n\right)
}\right)  \varphi\left(  C\right) \\
&  =E_{\sigma}\left(  P_{\left(  m,n\right)  }\left(  B^{\mathrm{T}}\otimes
A\right)  P_{\left(  m,n\right)  }\right)  _{\sigma}\varphi\left(  W\right)
+\left(  I_{2mn}+E_{\sigma}\mathcal{P}_{\left(  m,n\right)  }\right)
\varphi\left(  C\right) \\
&  =E_{\sigma}\mathcal{P}_{\left(  m,n\right)  }E_{\sigma}\mathcal{P}_{\left(
m,n\right)  }\varphi\left(  W\right)  +\left(  I_{2mn}+E_{\sigma}%
\mathcal{P}_{\left(  m,n\right)  }\right)  \varphi\left(  C\right)  ,
\end{align*}
where $E=B^{\mathrm{T}}\otimes A$ and $\mathcal{P}_{\left(  m,n\right)  }$ is
defined in\ (\ref{teq36}). It follows that%
\[
\left(  I_{2mn}-\left(  E_{\sigma}\mathcal{P}_{\left(  m,n\right)  }\right)
^{2}\right)  \varphi\left(  W\right)  =\left(  I_{2mn}+E_{\sigma}%
\mathcal{P}_{\left(  m,n\right)  }\right)  \varphi\left(  C\right)  .
\]
Similarly, equation (\ref{teq3}) and expression (\ref{teq105}) are,
respectively, equivalent to%
\[
\left(  I_{2mn}-E_{\sigma}\mathcal{P}_{\left(  m,n\right)  }\right)
=\varphi\left(  C\right)  ,
\]
and%
\[
\varphi\left(  \boldsymbol{X}\right)  =\frac{1}{2}\left(  I_{2mn}+E_{\sigma
}\mathcal{P}_{\left(  m,n\right)  }\right)  \varphi\left(  \boldsymbol{W}%
\right)  +\frac{1}{2}\varphi\left(  C\right)  .
\]
The remaining of the proof is similar to the proof of Theorem \ref{tth1} and
is omitted for brevity.

\textit{Proof of Item 3}. Clearly, equation (\ref{teq4}) has a unique solution
if and only if (\ref{teq5a}) is satisfied. By using Lemma \ref{tlm20},
(\ref{teq5a}) is equivalent to (\ref{teq6}) which is just the condition for
the existence of unique solution of equation (\ref{teq3}). The remaining of
the proof is similar to the proof of Item 3 of Theorem \ref{tth1}. The proof
is done.
\end{proof}

The advantage of Theorem \ref{tth2} is that to solve equation (\ref{teq3}), we
need only to consider the standard Stein equation (\ref{teq4}) which is again
in the form of (\ref{tstein}).

Similarly to the set $\mathscr{S}$ defined in (\ref{teqs}), if we let%
\[
\mathscr{T}=\left\{  x:\left(  I_{2mn}-E_{\sigma}\mathcal{P}_{\left(
m,n\right)  }\right)  x=0,\text{ }x\in\mathbf{R}^{2mn}\right\}  ,
\]
which is a linear subspace contained in $\mathbf{R}^{2mn}$ over the field
$\mathbf{R},$ then we can obtain the following corollary that parallels
Corollary \ref{tcoro1}.

\begin{corollary}
\label{teqcoro}Equation (\ref{teq3}) is solvable if and only if%
\[
\mathrm{rank}\left(  I_{mn}-\left(  A^{\mathrm{H}}B\right)  ^{\mathrm{T}%
}\otimes AB^{\mathrm{H}}\right)  =\mathrm{rank}\left(  \left[
\begin{array}
[c]{cc}%
I_{mn}-\left(  A^{\mathrm{H}}B\right)  ^{\mathrm{T}}\otimes AB^{\mathrm{H}} &
\left(  B^{\mathrm{T}}\otimes A\right)  \mathrm{vec}\left(  C^{\mathrm{H}%
}\right)  +\mathrm{vec}\left(  C\right)
\end{array}
\right]  \right)  .
\]
If the above relation is satisfied, then the degree of freedom (measured by
the maximal number of free real parameters), namely, $\dim\left(
\mathscr{T}\right)  ,$ in the solution is given by%
\[
\dim\left(  \mathscr{T}\right)  =mn-\mathrm{rank}\left(  I_{mn}-\left(
A^{\mathrm{H}}B\right)  ^{\mathrm{T}}\otimes AB^{\mathrm{H}}\right)  .
\]

\end{corollary}

\subsection{Solvability Conditions Based on Real Representation}

In this subsection, by using real representation of complex matrices, we
generalize the results in \cite{jw03laa} (namely, Proposition \ref{tlm4}) to
equation (\ref{teq3}).

\begin{theorem}
\label{tth4}Equation (\ref{teq3}) is solvable if and only if equation%
\begin{equation}
Y=A_{\sigma}Y^{\mathrm{T}}B_{\sigma}+C_{\sigma}, \label{teq23}%
\end{equation}
is solvable with unknown $Y$. More specifically,

\begin{enumerate}
\item If $X$ is a solution of equation (\ref{teq3}), then $Y=X_{\sigma}$ is a
solution of equation (\ref{teq23}).

\item If $Y$ is a real solution of (\ref{teq23}), then a solution of equation
(\ref{teq3}) is%
\[
X=\frac{1}{4}\left[
\begin{array}
[c]{cc}%
I_{m} & \mathrm{i}I_{m}%
\end{array}
\right]  \left(  Y+Q_{m}YQ_{n}\right)  \left[
\begin{array}
[c]{c}%
I_{n}\\
\mathrm{i}I_{n}%
\end{array}
\right]  .
\]

\item Equation (\ref{teq3}) has a unique solution $X^{\star}$ if equation
(\ref{teq23}) has a unique solution $Y^{\star}$, namely $1\notin
\lambda\left\{  \left(  B_{\sigma}^{\mathrm{T}}\otimes A_{\sigma}\right)
P_{\left(  2m,2n\right)  }\right\}  .$ Moreover, $Y^{\star}=X_{\sigma}^{\star
}.$
\end{enumerate}
\end{theorem}

\begin{proof}
\textit{Proof of Item 1.} Let $A=A_{1}+\mathrm{i}A_{2},B=B_{1}+\mathrm{i}%
B_{2}$ and $C=C_{1}+\mathrm{i}C_{2}$ where $A_{i},B_{i}$ and $C_{i},i=1,2$ are
real matrices. Then by substituting $X=X_{1}+\mathrm{i}X_{2}$ where
$X_{1},X_{2}\in\mathbf{R}^{m\times n}$ into (\ref{teq3}), we obtain
\[
\left\{
\begin{array}
[c]{l}%
X_{1}=\left(  A_{1}X_{1}^{\mathrm{T}}+A_{2}X_{2}^{\mathrm{T}}\right)
B_{1}-\left(  A_{2}X_{1}^{\mathrm{T}}-A_{1}X_{2}^{\mathrm{T}}\right)
B_{2}+C_{1},\\
X_{2}=\left(  A_{2}X_{1}^{\mathrm{T}}-A_{1}X_{2}^{\mathrm{T}}\right)
B_{1}+\left(  A_{1}X_{1}^{\mathrm{T}}+A_{2}X_{2}^{\mathrm{T}}\right)
B_{2}+C_{2},
\end{array}
\right.
\]
which is also equivalent to%
\[
\left[
\begin{array}
[c]{cc}%
X_{1} & X_{2}\\
X_{2} & -X_{1}%
\end{array}
\right]  =\left[
\begin{array}
[c]{cc}%
A_{1} & A_{2}\\
A_{2} & -A_{1}%
\end{array}
\right]  \left[
\begin{array}
[c]{cc}%
X_{1} & X_{2}\\
X_{2} & -X_{1}%
\end{array}
\right]  ^{\mathrm{T}}\left[
\begin{array}
[c]{cc}%
B_{1} & B_{2}\\
B_{2} & -B_{1}%
\end{array}
\right]  +\left[
\begin{array}
[c]{cc}%
C_{1} & C_{2}\\
C_{2} & -C_{1}%
\end{array}
\right]  .
\]
This indicates that $Y=X_{\sigma}$ satisfies (\ref{teq23}) in view of the real
representation (\ref{teqra}).

\textit{Proof of Item 2.} Let
\[
Y=\left[
\begin{array}
[c]{cc}%
Y_{1} & Y_{2}\\
Y_{3} & Y_{4}%
\end{array}
\right]  ,\text{ }Y_{i}\in\mathbf{R}^{m\times n},i=1,2,3,4,
\]
be a real solution of equation (\ref{teq23}). Note that $Q_{m}A_{\sigma}%
Q_{n}=A_{\sigma}$ (\cite{jw03laa}) and%
\[
Q_{m}^{-1}=\left[
\begin{array}
[c]{cc}%
0 & I_{m}\\
-I_{m} & 0
\end{array}
\right]  ^{-1}=\left[
\begin{array}
[c]{cc}%
0 & -I_{m}\\
I_{m} & 0
\end{array}
\right]  =Q_{m}^{\mathrm{T}}.
\]
Then by multiplying the left hand side and right hand side of (\ref{teq23}) by
respectively $Q_{m}$ and $Q_{n}$, we obtain%
\begin{align*}
Q_{m}YQ_{n}  &  =Q_{m}A_{\sigma}Y^{\mathrm{T}}B_{\sigma}Q_{n}+Q_{m}C_{\sigma
}Q_{n}\\
&  =Q_{m}A_{\sigma}Q_{n}Q_{n}^{-1}Y^{\mathrm{T}}Q_{m}^{-1}Q_{m}B_{\sigma}%
Q_{n}+Q_{m}C_{\sigma}Q_{n}\\
&  =A_{\sigma}Q_{n}^{\mathrm{T}}Y^{\mathrm{T}}Q_{m}^{\mathrm{T}}B_{\sigma
}+C_{\sigma}\\
&  =A_{\sigma}\left(  Q_{m}YQ_{n}\right)  ^{\mathrm{T}}B_{\sigma}+C_{\sigma}.
\end{align*}
This indicates that $Q_{m}YQ_{n}$ is also a solution of equation
(\ref{teq23}). Consequently,%
\begin{align*}
\mathcal{Y}  &  =\frac{1}{2}\left(  Y+Q_{m}YQ_{n}\right) \\
&  =\frac{1}{2}\left[
\begin{array}
[c]{cc}%
Y_{1}-Y_{4} & Y_{2}+Y_{3}\\
Y_{3}+Y_{2} & -\left(  Y_{1}-Y_{4}\right)
\end{array}
\right] \\
&  \triangleq\left[
\begin{array}
[c]{cc}%
X_{1} & X_{2}\\
X_{2} & -X_{1}%
\end{array}
\right]  ,
\end{align*}
is also a solution of equation (\ref{teq23}). Therefore, the following matrix%
\[
X\triangleq X_{1}+\mathrm{i}X_{2}=\frac{1}{2}\left[
\begin{array}
[c]{cc}%
I_{m} & \mathrm{i}I_{m}%
\end{array}
\right]  \mathcal{Y}\left[
\begin{array}
[c]{c}%
I_{n}\\
\mathrm{i}I_{n}%
\end{array}
\right]  ,
\]
is a solution of equation (\ref{teq3}).

\textit{Proof of Item 3}. We need only to show that if equation (\ref{teq23})
has a unique solution then equation (\ref{teq3}) also has a unique solution.
We show this by contradiction. Assume that equation (\ref{teq3}) has at least
two solutions, say, $X^{1}$ and $X^{2}$ with $X^{1}\neq X^{2}.$ Then,
according Item 1 of this theorem, both $X_{\sigma}^{1}$ and $X_{\sigma}^{2}$
are solutions of equation (\ref{teq23}). Clearly, we have $X_{\sigma}^{1}\neq
X_{\sigma}^{2}$ which yields a contradiction. The proof is finished.
\end{proof}

The following example indicates that the converse of Item 3 of Theorem
\ref{tth4} may also be true.

\begin{example}
\label{texamp3}Consider a linear equation in the form of (\ref{teq3}) with $A$
and $B$ given in (\ref{teqab}). Then equation (\ref{teq23}) has a unique
solution if and only if%
\begin{align*}
1  &  \notin\lambda\left\{  \left(  B_{\sigma}^{\mathrm{T}}\otimes A_{\sigma
}\right)  P_{\left(  2m,2n\right)  }\right\} \\
&  =\left\{  \pm2,\pm\sqrt{\left\vert \alpha\right\vert },\alpha
,\overline{\alpha},\pm\sqrt{2\alpha},\pm\sqrt{2\overline{\alpha}}\right\}  ,
\end{align*}
namely, $\left\vert \alpha\right\vert \neq1$ and $\alpha\neq\frac{1}{2}.$ On
the other hand, since $\lambda\left\{  AB^{\mathrm{H}}\right\}  =\{2,\alpha
\},$ the condition in (\ref{teq6}) is also equivalent to $\left\vert
\alpha\right\vert ^{2}\neq1$ and $\alpha\neq\frac{1}{2}.$ This implies, for
$A$ and $B$ given in (\ref{teqab}), that equation (\ref{teq23}) has a unique
solution if and only if equation (\ref{teq3}) has a unique solution.
\end{example}

However, we also cannot prove the above statement in general. We state it as a
conjecture which can be regarded as the generalization of Theorem \ref{tth9}
to equations (\ref{teq3}) and (\ref{teq23}).

\begin{conjecture}
If equation (\ref{teq3}) has a unique solution $X^{\star},$ then equation
(\ref{teq23}) has a unique solution $Y^{\star}$. Moreover, there holds
$Y^{\star}=X_{\sigma}^{\star}$.
\end{conjecture}

Even though the above conjecture is true, compared with the results in
Subsection \ref{tsec4.1}, there is no advantage by transforming equation
(\ref{teq3}) into equation (\ref{teq23}) via its real representation.

\subsection{Closed-Form and Iterative Solutions of Equation $X=AX^{\mathrm{H}%
}B+C$}

Based on Theorem \ref{tth2}, we can extend Corollary \ref{tcoro2} to equation
(\ref{teq3}) without a proof.

\begin{theorem}
\label{tth7}Let $A\in\mathbf{C}^{m\times n},B\in\mathbf{C}^{m\times n}%
,C\in\mathbf{C}^{m\times n}$ and
\[
h_{AB^{\mathrm{H}}}\left(  s\right)  =\det\left(  I_{m}-sAB^{\mathrm{H}%
}\right)  =\alpha_{m}+\sum\limits_{i=1}^{m}\alpha_{m-i}s^{i},\text{ }%
\alpha_{m}=1.
\]

\begin{enumerate}
\item If equation (\ref{teq3}) has a solution $X,$ then%
\[
Xh_{AB^{\mathrm{H}}}\left(  A^{\mathrm{H}}B\right)  =\sum\limits_{k=1}^{m}%
\sum\limits_{s=1}^{k}\alpha_{k}\left(  AB^{\mathrm{H}}\right)  ^{k-s}\left(
AC^{\mathrm{H}}B+C\right)  \left(  A^{\mathrm{H}}B\right)  ^{m-s}.
\]

\item If equation (\ref{teq3}) has a unique solution $X^{\star},$ then%
\[
X^{\star}=\left(  \sum\limits_{k=1}^{m}\sum\limits_{s=1}^{k}\alpha_{k}\left(
AB^{\mathrm{H}}\right)  ^{k-s}\left(  AC^{\mathrm{H}}B+C\right)  \left(
A^{\mathrm{H}}B\right)  ^{m-s}\right)  \left(  h_{AB^{\mathrm{H}}}\left(
A^{\mathrm{H}}B\right)  \right)  ^{-1}.
\]

\end{enumerate}
\end{theorem}

We next consider the Smith iteration associated with equation (\ref{teq3}).
Similar to (\ref{teq21}), this iteration is in the form of%
\begin{equation}
X_{k+1}=AX_{k}^{\mathrm{H}}B+C,\text{ \ }\forall X_{0}\in\mathbf{C}^{m\times
n}. \label{teq27}%
\end{equation}
The following result is about the convergence of the above iteration.

\begin{theorem}
The Smith iteration (\ref{teq27}) converges to the unique solution $X^{\star}$
of equation (\ref{teq3}) for arbitrary initial condition $X_{0}$ if and only
if%
\begin{equation}
\rho\left(  B^{\mathrm{H}}A\right)  <1. \label{teq24}%
\end{equation}
Moreover, the asymptotic exponential convergence rate is $-\ln\left(
\rho\left(  B^{\mathrm{H}}A\right)  \right)  .$
\end{theorem}

\begin{proof}
Similar to the proof of Theorem \ref{tth8}, we can show that if (\ref{teq27})
converges, we must have (\ref{teq24}). Therefore, we need only to show that if
(\ref{teq24}) is satisfied, the iteration in (\ref{teq27}) must converge to
the unique solution of equation (\ref{teq3}).

By using the real representation and Theorem \ref{tth5}, the iteration in
(\ref{teq27}) is equivalent to%
\begin{equation}
\left(  X_{k+1}\right)  _{\sigma}=A_{\sigma}\left(  \left(  X_{k}\right)
_{\sigma}\right)  ^{\mathrm{T}}B_{\sigma}+C_{\sigma}. \label{teq29}%
\end{equation}
According to Theorem \ref{tth5}, the above iteration converges to a constant
matrix independent of the initial condition if and only if $\rho\left(
A_{\sigma}B_{\sigma}^{\mathrm{T}}\right)  <1.$ According to Lemma \ref{tlm2},
we have%
\begin{align*}
\rho\left(  A_{\sigma}B_{\sigma}^{\mathrm{T}}\right)   &  =\rho\left(  \left(
AB^{\mathrm{H}}\right)  _{\phi}\right) \\
&  =\max\left\{  \rho\left(  AB^{\mathrm{H}}\right)  ,\rho\left(
\overline{AB^{\mathrm{H}}}\right)  \right\} \\
&  =\rho\left(  AB^{\mathrm{H}}\right)  .
\end{align*}
Hence, if (\ref{teq24}) is satisfied, iteration (\ref{teq29}) converges to a
constant matrix $X_{\infty}$ satisfying $\left(  X_{\infty}\right)  _{\sigma
}=A_{\sigma}\left(  \left(  X_{\infty}\right)  _{\sigma}\right)  ^{\mathrm{T}%
}B_{\sigma}+C_{\sigma}$, or equivalently, $X_{\infty}=AX_{\infty}^{\mathrm{H}%
}B+C.$ The proof is completed by observing that the equation (\ref{teq3}) has
a unique solution under condition (\ref{teq24}).
\end{proof}

\begin{remark}
\label{trm4}Under condition (\ref{teq24}), the unique solution of equation
(\ref{teq3}) is also the unique solution of equation (\ref{teq4}), namely,%
\[
X^{\star}=\sum\limits_{i=0}^{\infty}\left(  AA^{\mathrm{H}}\right)
^{i}\left(  AC^{\mathrm{H}}B+C\right)  \left(  B^{\mathrm{H}}B\right)  ^{i}.
\]
Therefore, Remark \ref{trm1} and Remark \ref{trm2} are also applicable to
equation (\ref{teq3}).
\end{remark}

We finally use a simple numerical example to end this section.

\begin{example}
\label{texamp4}Consider a linear equation in the form of (\ref{teq3}) with%
\[
A=\left[
\begin{array}
[c]{ccc}%
1 & 1+\mathrm{i} & 1\\
-2 & \mathrm{i} & -\mathrm{i}\\
1-\mathrm{i} & 0 & -1
\end{array}
\right]  ,\text{ \ }B=\left[
\begin{array}
[c]{ccc}%
\mathrm{i} & 1 & -1\\
0 & \mathrm{i} & 2+\mathrm{i}\\
1+\mathrm{i} & 3 & -\mathrm{i}%
\end{array}
\right]  ,
\]
and%
\[
C=\left[
\begin{array}
[c]{ccc}%
-5+\mathrm{i} & -4-\mathrm{i} & -5-12\mathrm{i}\\
2-\mathrm{i} & -4-2\mathrm{i} & 6+8\mathrm{i}\\
1+3\mathrm{i} & 15-5\mathrm{i} & -4-5\mathrm{i}%
\end{array}
\right]  .
\]
It is readily to verify that the eigenvalue set of $AB^{\mathrm{H}}$ satisfies
the condition in (\ref{teq6}). Then it follows from Theorem \ref{tth2} that
the matrix equation in (\ref{teq3}) has a unique solution which can be
obtained by using Theorem \ref{tth7} as follows:
\[
X=\left[
\begin{array}
[c]{ccc}%
1+3\mathrm{i} & -2 & 0\\
1 & 2-\mathrm{i} & 1\\
-2 & 2 & 2+\mathrm{i}%
\end{array}
\right]  .
\]

\end{example}

\section{\label{tsec7}A General Equation}

In this section, we point out that the results obtained above can be easily
extended to a general class of matrix equations having more terms on the right
hand side, say,\qquad%
\begin{equation}
X=\sum\limits_{i=0}^{N}A_{i}f\left(  X\right)  B_{i}+C, \label{teq37}%
\end{equation}
where $N\geq0,$ $A_{i}\in\mathbf{C}^{m\times m},B_{i}\in\mathbf{C}^{n\times
n},i\in\{0,1,\ldots,N\}$ and $C\in\mathbf{C}^{m\times n}$ are given, and
$X\in\mathbf{C}^{m\times n}$ is to be determined. In fact, by using
Proposition \ref{tpp1} and a similar technique used in Section \ref{tsec3}, we
can extend Theorem \ref{tth3} to equation (\ref{teq37}), as stated in the
following theorem.

\begin{theorem}
\label{tth10}Matrix equation (\ref{teq37}) is solvable if and only if the
following equation%
\begin{equation}
W=\sum\limits_{i=0}^{N}A_{i}\left(  \sum\limits_{k=0}^{N}f\left(
A_{k}f\left(  W\right)  B_{k}\right)  \right)  B_{i}+\sum\limits_{i=0}%
^{N}A_{i}f\left(  C\right)  B_{i}+C, \label{teq38}%
\end{equation}
is solvable with unknown $W$. More specifically:

\begin{enumerate}
\item If $X$ is a solution of equation (\ref{teq37}), then $W=X$ is a solution
of equation (\ref{teq38}).

\item If $\boldsymbol{W}$ is the general closed-form solution of equation
(\ref{teq38}), then the general closed-form solution of equation (\ref{teq37})
is given by
\[
\boldsymbol{X}=\frac{1}{2}\left(  \boldsymbol{W}+\sum\limits_{i=0}^{N}%
A_{i}f\left(  \boldsymbol{W}\right)  B_{i}+C\right)  .
\]

\item If $f\left(  X\right)  =\overline{X}$ or $f\left(  X\right)
=X^{\mathrm{H}},$ then equation (\ref{teq37}) has a unique solution $X^{\star
}$ if and only if equation (\ref{teq38}) has a unique solution $W^{\star}$. If
$f\left(  X\right)  =X^{\mathrm{T}},$ then equation (\ref{teq37}) has a unique
solution $X^{\star}$ if equation (\ref{teq38}) has a unique solution
$W^{\star}$.\ In both cases, $X^{\star}=W^{\star}.$
\end{enumerate}
\end{theorem}

\section{\label{tsec5}Conclusion}

This paper is concerned with solvability, existence of unique solution,
closed-form solution and numerical solution of matrix equation $X=Af\left(
X\right)  B+C$ with $f\left(  X\right)  =X^{\mathrm{T}},$ $f\left(  X\right)
=\overline{X}$ and $f\left(  X\right)  =X^{\mathrm{H}},$ where $X$ is a matrix
to be determined. It is established that the solvability of these equations
are equivalent to the solvability of some auxiliary standard Stein matrix
equations in the form of $W=\mathcal{A}W\mathcal{B}+\mathcal{C}$ where the
dimensions of the coefficient matrices $\mathcal{A},\mathcal{B}$ and
$\mathcal{C}$ are the same as dimensions of the coefficient matrices of the
original equation. Based on the auxiliary standard Stein matrix equation, the
conditions for solvability and the existence of unique solution are proposed.
Closed-form solutions are also obtained by inversion of a square matrix, which
is a generalization of the standard results on the standard Stein equation.
Numerical solutions are approximated by iterations which are generalizations
of the Smith iteration and accelerated Smith iteration associated with the
standard Stein equations. We should point out that the idea in this paper can
be readily adopted to study equations in the form of $AX-f\left(  X\right)
B=C$ with $f\left(  X\right)  =X^{\mathrm{T}},f\left(  X\right)  =\overline
{X}$ and $f\left(  X\right)  =X^{\mathrm{H}}.$

\bigskip

\section*{Acknowledgements}

The authors would like to thank the Editor and the anonymous reviewers for
their helpful comments and suggestions which have helped to improve the
quality of the paper.

This work is supported in part by the National Natural Science Foundation of
China under grant number 60904007 and by HKU CRCG 201007176243.

\end{document}